\documentclass[12pt]{amsart}
\usepackage{color,graphicx,amsmath,amsfonts,amssymb,amsxtra,setspace,xspace,subfig}
\usepackage[colorlinks=true]{hyperref}
\hypersetup{urlcolor=blue, citecolor=red, linkcolor=blue}

\theoremstyle{plain} 
\newtheorem{theorem}{Theorem} 
\newtheorem{corollary}{Corollary}
\newtheorem{lemma}{Lemma} 
\newtheorem{proposition}{Proposition}

\theoremstyle{definition} 
\newtheorem{definition}{Definition} 

\theoremstyle{remark}

\newcommand{\reo}{{\mathbb{R}}}

\newcommand{\ds}{\displaystyle}
\newcommand{\vp}{\varphi}

\newcommand{\cL}{{\mathcal{L}}}
\newcommand{\lp}[1]{\left( #1 \right)}
\newcommand{\ls}[1]{\left[ #1 \right]}
\newcommand{\lc}[1]{\left\{ #1 \right\}}

\newcommand{\vct}[1]{{\bf #1}}

\begin{document}

\title{1--Meixner random vectors}

\author{A. I. Stan*}
\address{ Department of Mathematics, The Ohio State University,
1465 Mount Vernon Avenue,
Marion, OH 43302, U.S.A.}
\email{stan.7@osu.edu*}
\author{F. Catrina}
 \address{ Department of Mathematics and  Computer Science,
St. John's University,
Queens, NY 11439, U.S.A.}
\email{catrinaf@stjohns.edu}
     
\maketitle

\begin{abstract}
A definition of $d$--dimensional $n$--Meixner random vectors is given first.
This definition involves the commutators of their
semi--quantum operators. After that we will focus on the $1$-Meixner random vectors,
and derive a system of $d$ partial differential equations satisfied by their
Laplace transform. We provide a set of necessary conditions for this system to
be integrable. We use these conditions to give a complete characterization of all
non--degenerate three--dimensional $1$--Meixner random vectors.
It must be mentioned that the three--dimensional case produces the first example in
which the components of a $1$--Meixner random vector cannot be reduced, via an injective
linear transformation, to three independent classic Meixner random variables.

\end{abstract}

\section{Introduction}
\label{intro}

Since its discovery in \cite{m34}, the class of Meixner random variables has been intensively
studied by many authors.
It seems that, among the six types of random variables belonging to this class (after a shifting
and re-scaling): Gaussian, Poisson, negative binomial, Gamma, two parameter hyperbolic secant,
and binomial, the two simplest should be the Gaussian and Poisson ones.
This is apparent from the fact that the principal Szeg\H{o}-Jacobi parameters,
$\{\omega_n\}_{n \geq 1}$, for these two
types of random variables, can be expressed as a linear function (with no constant term) of $n$,
\begin{eqnarray*}
\omega_n & = & tn,
\end{eqnarray*}
where $t$ is a nonnegative number,
while for the other four types, these parameters are quadratic functions (with no constant term)
of $n$,
\begin{eqnarray*}
\omega_n & = & \beta n^2 + (t - \beta)n,
\end{eqnarray*}
where either both $\beta$ and $t$ are non-negative, or $\beta$ is negative and $t$ is a negative integer
multiple of $\beta$ (hence $t$ is positive).\\
However, as it was pointed out in \cite{ps15}, from the point of view of the commutator between the
semi-quantum operators generated by each classic Meixner random variable, the two simplest are
the Gaussian and Gamma distributed ones.\\
There are some possible reasons why the Gamma distributed random
variables could be considered more ``basic" than the Poisson random variables. First of all, by
looking only at the simplicity of the principal Szeg\H{o}-Jacobi parameters, $\{\omega_n\}_{n \geq 1}$,
means to ignore the importance of the secondary Szeg\H{o}-Jacobi parameters, $\{\alpha_n\}_{n \geq 0}$,
which for all classic Meixner random variables are linear functions of $n$:
\begin{eqnarray*}
\alpha_n & = & \alpha n + \alpha_0,
\end{eqnarray*}
for all $n \geq 0$, where $\alpha$ and $\alpha_0$ are both fixed real numbers.
For the Gamma and Gaussian distributed random variables, the fixed real numbers $\alpha$ and $\beta$ are
not independent of each other, but they are linked by the relation:
\begin{eqnarray*}
\alpha^2 & = & 4\beta.
\end{eqnarray*}
The Szeg\H{o}-Jacobi parameters make sense in the one-dimensional case. It has been proposed in
\cite{aks03} and \cite{aks04} that a good replacement of the Szeg\H{o}-Jacobi parameters, in the
multi-dimensional case, is given by the quantum operators: creation, preservation, and annihilation
operators. For polynomially symmetric random vectors, the preservation operators vanish,
see \cite{aks03}.
For this reason, it is much easier to study the polynomially symmetric random vectors
than the non-symmetric ones. In the symmetric case, the multiplication operator generated by
each random variable (assumed to have finite
moments of all orders) is the sum of only two operators: creation and annihilation operators.
For non-symmetric random vectors, the preservation operators play a significant role.
In order to capture the effect of these operators, and still treat the multiplication operator
generated by each random variable as a sum of only two operators, in \cite{ps14}, each preservation
operator was split into two halves. One half was added to the corresponding creation operator, while
the other half was added to the annihilation operator. In this way, the semi-quantum operators: semi-creation and
semi-annihilation operators were defined, and the multiplication operator generated by each random variable
can also be written as a sum of only two semi-quantum operators. Moreover, by splitting the preservation
operators into two equal parts, the fact that each semi-creation operator is the polynomial dual
of its corresponding semi-annihilation operator was preserved. In this way, the non-symmetric random vectors
can be treated similarly to the symmetric ones.\\
It must also be mentioned that using the semi-quantum operators, rather than the quantum operators,
it was possible not only to give an elegant way to describe the Gamma and Gaussian random variables, using only one condition,
but to also come up with the definition of a countable family of Meixner classes. Each class
is determined by the number of nested commutators involving the semi-quantum operators,
see \cite{ps15}. Thus for example,
the Gamma and Gaussian use only one commutator, while the class of classic Meixner random variables is described by two nested commutators.\\
In \cite{ps15} the non-degenerate two dimensional $1$-Meixner random vectors were characterized.
The method used in that paper was particular to the two dimensional case, and did not give any indication
about how to proceed in the multi-dimensional case. In the two-dimensional case, the components of a $1$-Meixner random vector can be reduced,
via an injective affine transformation, to two independent classic Gamma or Gaussian random variables.\\
In this paper, we find first a system of $d$ linear partial differential equations satisfied by the Laplace transform of a non-degenerate
$d$--dimensional $1$--Meixner random vector. Then we find a set of necessary conditions for the integrability of this system of linear partial
differential equations. Finally, we give a complete characterization of all non--degenerate three--dimensional $1$--Meixner random
vectors. The important thing that will appear in our characterization is the fact that there are non--degenerate three--dimensional Meixner random
vectors whose components cannot be reduced, via any injective affine transformation, to three independent random variables.
This fact is of great importance, because it shows the power of the quantum and semi-quantum operators as natural
extensions of the classic Szeg\H{o}--Jacobi parameters to the multi-dimensional case.
\\
The paper is structured as follows.
In section 2, we give a minimal background of quantum and semi-quantum operators. In section 3, we review the Meixner random variables. In section 4, we present a set of simplifying assumptions and consistency conditions.
In section 5, we present a system of partial differential equations which has to be satisfied by
the Laplace transform of  any $d$-dimensional $1$-Meixner random vector,
for all natural numbers $d$.
In section 6, we find a necessary condition for the existence of solutions for this system.
In section 7, we describe all solutions of the system in the particular case $d=3$.
Finally, in the Appendix, we present calculations of the Laplace transforms of measures, that match the solutions found in the three--dimensional case.

\section{Background}
\label{S:bckg}

Throughout this paper we consider $d$ random variables, $X_1$, $X_2$, $\dots$, $X_d$, having finite moments
of all orders, and defined on the same probability space $(\Omega$, ${\mathcal F}$, $P)$, where $d$ is a fixed
natural number. We define the space:
\begin{eqnarray*}
F & := & \{f(X_1, X_2, \dots, X_d) \mid f \ {\rm is \ polynomial}\},
\end{eqnarray*}
and call it the space of all {\em polynomial random variables in} $X_1$, $X_2$, $\dots$, $X_d$. The polynomials $f$ in this definition are polynomials of $d$ variables with complex coefficients.\\
For each non-negative integer $n$, we define the space:
\begin{eqnarray*}
F_n & := & \{f(X_1, X_2, \dots, X_d) \mid f \ {\rm is \ polynomial \ of \ degree \ at \ most} \ n\}.
\end{eqnarray*}
Since $X_1$, $X_2$, $\dots$, $X_d$ have finite moments of all orders, we have:
\begin{eqnarray*}
{\mathbb C} \equiv F_0 \subseteq F_1 \subseteq F_2 \subseteq \cdots \subseteq F \subseteq L^2(\Omega, {\mathcal F}, P).
\end{eqnarray*}
Moreover, since for each $n \geq 0$, $F_n$ is a finite dimensional vector space, we conclude that
$F_n$ is a closed subspace of $L^2(\Omega, {\mathcal F}, P)$.\\
Since the spaces $\{F_n\}_{n \geq 0}$ are closed and contained one into another, we can orthogonalize them with
respect to the inner product, $\langle \cdot$, $\cdot \rangle$, of the space $L^2(\Omega, {\mathcal F}, P)$.
Thus, we define:
\begin{eqnarray*}
G_0 & := & F_0,
\end{eqnarray*}
and for all $n \geq 1$,
\begin{eqnarray*}
G_n & := & F_n \ominus F_{n - 1},
\end{eqnarray*}
that means $G_n$ is the orthogonal complement of $F_{n - 1}$ in $F_n$.
For each $n \geq 0$, we call $G_n$ the {\em $n$-th homogenous chaos space} generated by $X_1$, $X_2$,
$\dots$, $X_d$. We also call every random variable $f(X_1$, $X_2$, $\dots$, $X_d)$ in $G_n$ a {\em homogenous
polynomial random variable of degree} $n$ (here, the word ``homogenous" does not
have the classic meaning that all terms have the same degree).\\
Define the spaces $F_{-1} = G_{-1} = \{0\}$, where $\{0\}$ denotes the null space.\\
We change now the way that we view the random variables $X_1$, $X_2$, $\dots$, $X_d$, by regarding them as
the multiplication operators that they generate. That means, for each $i \in \{1$, $2$, $\dots$, $d\}$,
we consider the linear operator going from $F$ to $F$, defined by:
\begin{eqnarray*}
f\left(X_1, X_2, \dots, X_d\right) & \mapsto & X_if\left(X_1, X_2, \dots, X_d\right).
\end{eqnarray*}
For all $i \in \{1$, $2$, $\dots$, $d\}$, we denote this operator by $X_i$.\\
In what follows, instead of $f(X_1$, $X_2$, $\dots$, $X_d)$, we write briefly $f$.\\
Regarding the multiplication
operators $X_1$, $X_2$, $\cdots$, $X_d$, we have
the following lemma that can be found in \cite{an02}, Theorem 1, page 6, (see also \cite{aks03}, Lemma 2.1., page 487).
\begin{lemma}\label{def_lemma}
For all $i \in \{1$, $2$, $\dots$, $d\}$ and all non--negative integers $n$, we have:
\begin{eqnarray*}
X_iG_n & \bot & G_k,
\end{eqnarray*}
for all $k \neq n - 1$, $n$, $n + 1$, where ``$\bot$" means ``orthogonal to".
\end{lemma}
From this lemma, we conclude that, for all $i \in \{1$, $2$, $\dots$, $d\}$ and all $n \geq 0$, we have:
\begin{eqnarray*}
X_iG_n & \subseteq & G_{n - 1} \oplus G_n \oplus G_{n + 1}.
\end{eqnarray*}
That means if $f \in G_n$, there exist and are unique three homogenous polynomial random variables:
$f_{n - 1, i} \in G_{n - 1}$, $f_{n, i} \in G_n$, and $f_{n + 1, i} \in G_{n + 1}$, such that:
\begin{eqnarray*}
X_if & = & f_{n - 1, i} + f_{n, i} + f_{n + 1, i}.
\end{eqnarray*}
We define the following linear operators:
\begin{eqnarray*}
D_n^-(i) : G_n & \to & G_{n - 1},\\
D_n^-(i)f & := & f_{n - 1, i},
\end{eqnarray*}
and call $D_n^-(i)$ an {\em annihilation operator}, since it decreases the degree of a homogenous polynomial
by one unit,
\begin{eqnarray*}
D_n^0(i) : G_n & \to & G_n,\\
D_n^0(i)f & := & f_{n, i},
\end{eqnarray*}
and call $D_n^0(i)$ a {\em preservation operator}, because it preserves the degree of a homogenous polynomial,
and
\begin{eqnarray*}
D_n^+(i) : G_n & \to & G_{n + 1},\\
D_n^+(i)f & := & f_{n + 1, i},
\end{eqnarray*}
and call $D_n^+(i)$ a {\em creation operator}, due to the fact that it increases the degree of a homogenous polynomial by one unit.\\
Lemma \ref{def_lemma} can be written now:
\begin{lemma}
For all $1 \leq i \leq d$ and all $n \geq 0$, we have:
\begin{eqnarray*}
X_i|G_n & = & D_n^-(i) + D_n^0(i) + D_n^+(i),
\end{eqnarray*}
where $X_i|G_n$ is the restriction of the multiplication operator $X_i$ to the space $G_n$.
\end{lemma}
We extend now, by linearity, the definition of the annihilation, preservation, and creation operators to the space $F$ of all polynomial random variables, in the following way. If $f \in F$, then there exist
and are unique homogenous polynomial random variables $f_0 \in G_0$, $f_1 \in G_1$, $f_2 \in G_2$, $\dots$, with only finitely many of them being different from zero, such that:
\begin{eqnarray*}
f & = & f_0 + f_1 + f_2 + \cdots.
\end{eqnarray*}
We define the $i$-th {\em annihilation operator} by:
\begin{eqnarray*}
a^-(i)f & = & D_0^-(i)f_0 + D_1^-(i)f_1 + D_2^-(i)f_2 + \cdots,
\end{eqnarray*}
$i$-th {\em preservation operator} by:
\begin{eqnarray*}
a^0(i)f & = & D_0^0(i)f_0 + D_1^0(i)f_1 + D_2^0(i)f_2 + \cdots,
\end{eqnarray*}
and
$i$-th {\em creation operator} by:
\begin{eqnarray*}
a^+(i)f & = & D_0^+(i)f_0 + D_1^+(i)f_1 + D_2^+(i)f_2 + \cdots.
\end{eqnarray*}
Lemma~\ref{def_lemma} becomes now:
\begin{lemma}
For all $i \in \{1$, $2$, $\dots$, $d\}$, we have:
\begin{eqnarray*}
X_i & = & a^-(i) + a^0(i) + a^+(i),
\end{eqnarray*}
where the domain of $X_i$, $a^-(i)$, $a^0(i)$, and $a^+(i)$ is considered to be
the space $F$ of all polynomial random variables.
\end{lemma}
We call the operators: $\{a^-(i)\}_{1 \leq i \leq d}$, $\{a^0(i)\}_{1 \leq i \leq d}$,
and $\{a^+(i)\}_{1 \leq i \leq d}$ the {\em joint quantum operators of}
$X_1$, $X_2$, $\dots$, $X_d$. It is not hard to see that, for all
$i \in \{1$, $2$, $\dots$, $d\}$, we have:
\begin{eqnarray*}
\left(a^+(i)\right)^* & = & a^-(i)
\end{eqnarray*}
and
\begin{eqnarray*}
\left(a^0(i)\right)^* & = & a^0(i),
\end{eqnarray*}
where the above duality is a {\em polynomial duality}, that means,
for all $f$ and $g$ in $F$, we have:
\begin{eqnarray*}
\langle a^+(i)f, g \rangle & = & \langle f, a^-(i)g \rangle
\end{eqnarray*}
and
\begin{eqnarray*}
\langle a^0(i)f, g \rangle & = & \langle f, a^0(i)g \rangle.
\end{eqnarray*}
It is clear that for all $i$ and $j$ in $\{1$, $2$, $\dots$, $d\}$,
the multiplication operators by $X_i$ and $X_j$ commute, that means:
\begin{eqnarray*}
X_iX_j & = & X_jX_i.
\end{eqnarray*}
It was shown in \cite{aks03} and \cite{aks04}, that the commutativity of the multiplication
operators by $X_i$ and $X_j$ is equivalent, in terms of the commutators of the joint quantum operators,
to the following set of rules:
\begin{eqnarray}
\left[a^-(i), a^-(j)\right] & = & 0, \label{--}
\end{eqnarray}
\begin{eqnarray}
\left[a^-(i), a^0(j)\right] & = & \left[a^-(j), a^0(i)\right], \label{-0}
\end{eqnarray}
\begin{eqnarray}
\left[a^0(i), a^0(j)\right] & = & \left[a^-(j), a^+(i)\right] - \left[a^-(i), a^+(j)\right], \label{-0+}
\end{eqnarray}
\begin{eqnarray}
\left[a^0(i), a^+(j)\right] & = & \left[a^0(j), a^+(i)\right], \label{0+}
\end{eqnarray}
and
\begin{eqnarray}
\left[a^+(i), a^+(j)\right] & = & 0. \label{++}
\end{eqnarray}
We refer to the commutation rules (\ref{--}), (\ref{-0}), (\ref{-0+}),
(\ref{0+}), and (\ref{++}), as the {\em axioms of Commutative Probability}.\\
For all $i \in \{1$, $2$, $\dots$, $d\}$, we define the linear operators:
\begin{eqnarray*}
U_i, V_i : F & \to & F,
\end{eqnarray*}
\begin{eqnarray*}
U_i & = & a^-(i) + \frac{1}{2}a^0(i)
\end{eqnarray*}
and
\begin{eqnarray*}
V_i & = & a^+(i) + \frac{1}{2}a^0(i).
\end{eqnarray*}
For all $i \in \{1$, $2$, $\dots$, $d\}$, we call $U_i$ a {\em semi-annihilation operator},
and $V_i$ a {\em semi--creation operator}. We also call
$\{U_i\}_{1 \leq i \leq d}$ and $\{V_i\}_{1 \leq i \leq d}$ the {\em joint semi-quantum operators}
generated by $X_1$, $X_2$, $\dots$, $X_d$.\\
It is now clear that, for all $i \in \{1$, $2$, $\dots$, $d\}$, we have:
\begin{eqnarray*}
X_i & = & U_i + V_i
\end{eqnarray*}
and
\begin{eqnarray*}
V_i^* & = & U_i,
\end{eqnarray*}
where the above duality is, as before, only a polynomial duality.\\
As it was shown in \cite{ps14}, the axioms of Commutative Probability,
can be written now in terms of the commutators involving the joint
semi-quantum operators as any one of the following three equivalent statements:

\begin{enumerate}

\item For all $(i,j) \in \{1$, $2$, $\dots$, $d\}^2$, we have:
\begin{eqnarray*}
\left[U_i, X_j\right] & = & \left[U_j, X_i\right].
\end{eqnarray*}

\item For all $(i,j) \in \{1$, $2$, $\dots$, $d\}^2$, we have:
\begin{eqnarray*}
\left[X_i, V_j\right] & = & \left[X_j, V_i\right].
\end{eqnarray*}

\item For all $(i,j) \in \{1$, $2$, $\dots$, $d\}^2$,
the operators $[U_i$, $X_j]$ and $[X_i$, $V_j]$ are polynomially self--adjoint.
\end{enumerate}
Let us see what this general theory becomes in the one dimensional case, $d = 1$.
For $d = 1$, there is no need to use subscripts since we are dealing with only one random
variable, $X$, one creation, $a^+$, one preservation, $a^0$, one annihilation, $a^-$,
one semi--creation, $V$, and one semi--annihilation operator, $U$.\\
For all $n \geq 0$, since the co--dimension of the space $F_{n - 1}$
(spanned by $1$, $X$, $\dots$, $X^{n - 1}$) into $F_n$ (spanned by $1$, $X$, $\dots$, $X^{n - 1}$,
$X^n$) is at most $1$, the homogenous chaos space $G_n = F_n \ominus F_{n - 1}$ has dimension
at most $1$. If the random variable $X$ takes on only a finite number, $k$, of different values,
with positive probability, then we have:
\begin{eqnarray*}
\dim(G_n) & = & \left\{\begin{array}{ccc} 1 & {\rm if} & n \leq k - 1\\
0 & {\rm if} & n \geq k \end{array}\right.,
\end{eqnarray*}
where ``$\dim$" denotes  the dimension.
If the probability distribution, of the random variable $X$, has an infinite support, then, for all
$n \geq 0$, we have:
\begin{eqnarray*}
\dim(G_n) & = & 1.
\end{eqnarray*}
If $\dim(G_n) = 1$, then there exists a unique polynomial $f_n \in G_n$, having the leading coefficient equal to $1$.
Because $XG_n \subseteq G_{n + 1} + G_n + G_{n - 1}$, there exist $\alpha_n$ and $\omega_n$ real numbers, such that:
\begin{eqnarray*}
Xf_n(X) & = & f_{n + 1}(X) + \alpha_nf_n(X) + \omega_nf_{n - 1}(X).
\end{eqnarray*}
For $n = 0$, because $f_{-1} = 0$, we can choose $\omega_0$ as we please.
The real numbers $\{\alpha_n\}_{n \geq 0}$ and
$\{\omega_n\}_{n \geq 1}$ are called the {\em Szeg\H{o}--Jacobi parameters of} $X$.\\
In the case $d = 1$, for all $n \geq 0$, we have:
\begin{eqnarray*}
a^-f_n & = & \omega_nf_{n - 1},
\end{eqnarray*}
\begin{eqnarray*}
a^0f_n & = & \alpha_nf_n,
\end{eqnarray*}
\begin{eqnarray*}
a^+f_n & = & f_{n + 1},
\end{eqnarray*}
\begin{eqnarray*}
Uf_n & = & \frac{\alpha_n}{2}f_n + \omega_nf_{n - 1},
\end{eqnarray*}
and
\begin{eqnarray*}
Vf_n & = & f_{n + 1} + \frac{\alpha_n}{2}f_n.
\end{eqnarray*}

\section{$d$-dimensional $n$-Meixner random vectors: definition and general properties}
\label{S:GenDef}
We review now the classic Meixner random variables.
\begin{definition}
A real valued random variable $X$, having finite moments of all orders,
is called a {\em classic Meixner random variable} if
its Szeg\H{o}-Jacobi parameters are of the form:
\begin{eqnarray*}
\alpha_n & = & \alpha n + \alpha_0
\end{eqnarray*}
and
\begin{eqnarray*}
\omega_n & = & \beta n^2 + (t - \beta)n,
\end{eqnarray*}
for all $n \geq 1$, where $\alpha$, $\alpha_0$, $\beta$, and $t$ are fixed
real numbers such that one of the two possible scenarios happens:
\begin{enumerate}
\item
$\beta \geq 0$ and $t \geq 0$.

\item
$\beta < 0$ and $t \in -{\mathbb N}\beta$.
\end{enumerate}
\end{definition}
\noindent Of course, if $t = 0$, then $\omega_1 = 0$, and so $X$ is a constant random variable,
since the $L^2$-norm of the $1$-degree monic orthogonal polynomial is $\parallel f_1\parallel_2^2 = \omega_1 = 0$.
This case is not interesting, and we are going to assume that $t > 0$.\\
We may also assume that $\alpha \geq 0$ since otherwise, we can replace $X$ by $-X$.\\
There are six types of Meixner random variables:
\begin{itemize}
\item If $\alpha = \beta = 0$, then $X$
is a {\em Gaussian} random variable, i.e., a continuous random variable given by the density function
\begin{eqnarray*}
f(x) & = & \frac{1}{\sqrt{2\pi t}}e^{-(x - \alpha_0)^2/(2t)}.
\end{eqnarray*}

\item If $\beta = 0$ and $\alpha \neq 0$, then $X$ is a shifted and re-scaled {\em Poisson}
random variable, i.e.,
\begin{eqnarray*}
\mu_X & = & \sum_{k =
0}^{\infty}\frac{\lambda^k}{k!}e^{-\lambda}\delta_{\alpha(k -
\lambda) + \alpha_0},
\end{eqnarray*}
where $\lambda := t/\alpha^2$.

\item If $\beta > 0$ and $\alpha^2 > 4\beta$, then $X$ is a
shifted {\em Pascal (negative binomial)} random variable, i.e.,
\begin{eqnarray*}
\mu_X & = & \sum_{k = 0}^{\infty}\frac{\Gamma(r +
k)}{k!\Gamma(r)}p^r(1 - p)^k\delta_{k - [2t/(\alpha + d)] +
\alpha_0},
\end{eqnarray*}
where $d := \sqrt{\alpha^2 - 4\beta}$, $p := 2d/(\alpha + d)$, $r :=
t/\beta$.

\item If $\beta > 0$ and $\alpha^2 = 4\beta$, then $X$ is a shifted
and re--scaled {\em Gamma distributed} random variable with shift parameter
$2t/\alpha$ and scaling parameter $\alpha/2$, i.e.,
\begin{eqnarray*}
f(x) & = &
\frac{2^{2t/\alpha}}{\alpha^{2t/\alpha}\Gamma(2t/\alpha)}x^{(2t/\alpha)
- 1}e^{-2x/\alpha}1_{(0, \infty)}.
\end{eqnarray*}

\item If $\beta > 0$ and $\alpha^2 < 4\beta$, then up to a
translation, $X$ is a {\em two parameter hyperbolic
secant} random variable:
\begin{eqnarray*}
f(x) & = & ce^{2\theta x/\gamma}\left|\Gamma(k + ix\gamma)\right|^2,
\end{eqnarray*}
where $\gamma := \sqrt{4\beta - \alpha^2}$ and $\gamma + i\alpha =
re^{i\theta}$, with $-\pi/2 < \theta < \pi/2$, $k := 2t/(r\gamma)$.

\item If $\beta < 0$, then $t \in -{\mathbb N}\beta$, and in this case,
up to a shifting and re-scaling, $X$ is a {\em binomial} random variable:
\begin{eqnarray*}
\mu_X & = & \sum_{k = 0}^n\binom{n}{k}p^k(1 - p)^{n - k}\delta_k,
\end{eqnarray*}
where $n := -t/\beta$, $p := (1/2) \pm (1/2)\sqrt{c/(4 + c)}$, and $c :=
-\alpha^2/\beta \geq 0$.

\end{itemize}

In \cite{ps14}, it was shown that, the shifted and re--scaled Gamma distributed random variables
(Meixner with $\alpha^2 = 4\beta > 0$) and the Gaussian random variables (Meixner with
$\alpha = \beta = 0$), are exactly those random variables $X$, having finite
moments of all orders, for which the commutator
between the semi--annihilation operator $U$ and $X$ is of the form:
\begin{eqnarray}
\left[U, X\right] & = & bX + cI, \label{UX}
\end{eqnarray}
where $b$ and $c$ are real numbers.
Using the polynomial duality between $U$ and $V$, where $V$ denotes the
semi-creation operator of $X$, this condition is equivalent to:
\begin{eqnarray}
\left[X, V\right] & = & bX + cI. \label{VX}
\end{eqnarray}
We must mention that the equality (\ref{UX}) makes sense since one of the axioms
of Commutative Probability says that $[U$, $X]$ is a polynomially self--adjoint operator.
The shifted multiplication operator $bX + cI$ from the right of (\ref{UX}) is clearly
self-adjoint. If $[U$, $X]$ where not self adjoint, then equality (\ref{UX}) would have
been impossible.\\
It was also shown in \cite{ps16}, that the classic Meixner random variables (all six of them) are
exactly those random variables $X$, having finite moments of all order, for which the double
commutator $[[U$, $X]$, $X]$ is of the form:
\begin{eqnarray}
\left[\left[U, X\right], X\right] & = & bV - bU \label{bV-bU}\\
& = & b(X - 2U),\nonumber
\end{eqnarray}
where $b$ is a real number. One may wonder why the right-hand side of (\ref{bV-bU}) looks much different
than the right-hand side of \eqref{UX}. The reason is that while the commutator $[U$, $X]$,
from the left-hand side of (\ref{UX}), is polynomially self-adjoint, the double commutator
$[[U$, $X]$, $X]$, from the left-hand side of (\ref{bV-bU}) is polynomially anti-self-adjoint.
Thus, the coefficients of $U$ and $V$, in the right-hand side of (\ref{bV-bU}), must be opposite one to
another. \\
Observe that when we have an odd number of nested commutators, we obtain a polynomially self-adjoint
operator, while an even number of nested commutators creates a polynomially anti-self-adjoint operator.
The following definition of different types of Meixner random vectors was proposed in \cite{ps15}.
\begin{definition}
Let $X_1$, $X_2$, $\dots$, $X_d$ be $d$ random variables having
finite moments of all orders. Let $n$ be a natural number.
We say that $(X_1$, $X_2$, $\dots$, $X_d)$ is a {\em $d$--dimensional $n$--Meixner random vector} if:
\begin{itemize}
\item If $n$ is odd, then
for all $i$, $i_1$, $i_2$, $\dots$, $i_n$ in $\{1$, $2$, $\dots$, $d\}$, we have:
\begin{eqnarray*}
\left[\cdots \left[\left[U_{i}, X_{i_1}\right], X_{i_2}\right], \cdots, X_{i_n}\right] & = & \sum_{j = 1}^d
b_{i, i_1i_2 \dots i_n, j}X_j + c_{i, i_1i_2 \dots i_n}I,
\end{eqnarray*}
for some real numbers $b_{i, i_1i_2 \dots i_n, j}$ and $c_{i, i_1i_2 \dots i_n}$, $1 \leq j \leq d$.

\item If $n$ is even, then
for all $i$, $i_1$, $i_2$, $\dots$, $i_n$ in $\{1$, $2$, $\dots$, $d\}$, we have:
\begin{eqnarray*}
\left[\cdots \left[\left[U_{i}, X_{i_1}\right], X_{i_2}\right], \cdots, X_{i_n}\right] & = & \sum_{j = 1}^d
b_{i, i_1i_2 \dots i_n, j}(V_j - U_j)\\
 & = & \sum_{j = 1}^db_{i, i_1i_2 \dots i_n, j}(X_j - 2U_j),
\end{eqnarray*}
for some real numbers $b_{i, i_1i_2 \dots i_n, j}$, $1 \leq j \leq d$.
\end{itemize}
Moreover, we say that the random vector $(X_1$, $X_2$, $\dots$, $X_d)$ is {\em non--degenerate}
if the operators $I$, $X_1$, $X_2$, \dots, $X_d$ are linearly independent. If $\mu$ denotes the
joint probability distribution of $X_1$, $X_2$, $\dots$, $X_d$, and the space of
all polynomial
functions of $d$ variables, $F(x_1$, $x_2$, $\dots$, $x_d)$, is dense in $L^2({\mathbb R}^d$, $\mu)$, then
the non-degeneracy condition is equivalent to the fact that
$1$, $X_1$, $X_2$, $\dots$, $X_d$ are linearly independent as random variables, where $1$ denotes the constant random variable equal to $1$.
\end{definition}
According to this definition, the shifted and re--scaled Gamma and Gaussian distributed random variables
form the one--dimensional $1$--Meixner random vectors (variables).
The classic Meixner random variables (all six types of them) are precisely the one--dimensional $2$--Meixner random vectors.
We would like to stress that, this definition, that uses the semi--quantum operators, permits us to include
all the six types of classic Meixner random variables. It also allows us to keep the number of commutator conditions to a minimum. Moreover, it allows us to study not only $3$--Meixner, $4$--Meixner, $\dots$ random variables, but also random vectors.\\
In \cite{ps14} is was shown how this definition, employing the commutator between $U$ and $X$, can be used
effectively to recover first the moments and then the probability distributions of the $1$-Meixner
random variables. In \cite{ps16}, it was shown how the definition using the double commutators can be applied
to recover the probability distributions of all the classic Meixner random variables.
In \cite{ps15}, the first step in moving from one dimension to two dimensions was achieved, and all
non-degenerate two dimensional $1$-Meixner random vectors were described.
In this paper, we will find a system of differential equations satisfied by the Laplace transform of each non--degenerate
$d$--dimensional $1$--Meixner random vector, for every finite dimension $d$, and characterize all
the non-degenerate three dimensional $1$-Meixner random vectors.

\section{$d$-dimensional $1$-Meixner random vectors: simplifying assumptions and consistency
conditions}

In this section we make some simplifying assumptions that will ease our work in the next section.
We also establish some important consistency conditions. The following is a very simple fact to check.
\begin{proposition} \label{transmissibility_proposition}
Let $X_1$, $X_2$, $\dots$, $X_d$ be $d$-random variables defined on the same probability space
$(\Omega$, ${\mathcal F}$, $P)$ and having finite moments of all orders. Let $\{U_i\}_{1 \leq i \leq d}$
and $\{V_i\}_{1 \leq i \leq d}$ be the joint semi-annihilation and semi-creation operators, respectively,
generated by $X_1$, $X_2$, $\dots$, $X_d$.
If $A = (a_{i, j})_{1 \leq i, j \leq d}$ is a $d \times d$ invertible matrix with real entries, and
$b = (b_i)_{1 \leq i \leq d}$ is a vector in ${\mathbb R}^d$, then if we define the random variables:
\begin{eqnarray*}
X_1' & := & a_{1, 1}X_1 + a_{1, 2}X_2 + \cdots + a_{1, d}X_d + b_1 \nonumber\\
X_2' & := & a_{2, 1}X_1 + a_{2, 2}X_2 + \cdots + a_{2, d}X_d + b_2 \nonumber\\
\vdots & \vdots & \vdots \\
X_d' & := & a_{d, 1}X_1 + a_{d, 2}X_2 + \cdots + a_{d, d}X_d + b_d,
\end{eqnarray*}
then $X_1'$, $X_2'$, $\dots$, $X_d'$ have finite moments of all orders, and their
joint semi-quantum operators satisfy:
\begin{eqnarray*}
U_1' & = & a_{1, 1}U_1 + a_{1, 2}U_2 + \cdots + a_{1, d}U_d + b_1 \nonumber\\
U_2' & = & a_{2, 1}U_1 + a_{2, 2}U_2 + \cdots + a_{2, d}U_d + b_2 \nonumber\\
\vdots & \vdots & \vdots \\
U_d' & = & a_{d, 1}U_1 + a_{d, 2}U_2 + \cdots + a_{d, d}U_d + b_d,
\end{eqnarray*}
and similar formulas hold for their semi-creation operators.
Moreover, if $(X_1$, $X_2$, $\dots$, $X_d)$ is non-degenerate, then
$(X_1'$, $X_2'$, $\dots$, $X_d')$ is also non-degenerate.
\end{proposition}
Let us assume now that $(X_1$, $X_2$, $\dots$, $X_d)$ is a non-degenerate $d$-dimensional
$1$-Meixner random vector. That means, there exist two finite sequences of real numbers:
$\{\alpha_{i, j, k}\}_{1 \leq i, j, k \leq d}$ and $\{\beta_{i, j}\}_{1 \leq i, j \leq d}$, such that,
for all $(i$, $j) \in \{1$, $2$, $\dots$, $d\}^2$, we have:
\begin{eqnarray*}
\left[U_i, X_j\right] & = & \sum_{k = 1}^d\alpha_{i, j, k}X_k + \beta_{i, j}I,
\end{eqnarray*}
where $I$ denotes the identity operator on the space $F$ of all polynomial random variables
in $X_1$, $X_2$, $\dots$, $X_d$.\\
It follows now from Proposition \ref{transmissibility_proposition}, that if we apply an invertible affine transformation $T : {\mathbb R}^d \to {\mathbb R}^d$,
\begin{eqnarray*}
Ty & = & Ay + b,
\end{eqnarray*}
where $A$ is a $d \times d$ invertible matrix, and $b$ a vector in ${\mathbb R}^d$, to the random vector
$X := (X_1$, $X_2$, $\dots$, $X_d)$, then the obtained random vector:
\begin{eqnarray*}
X' & := & AX + b
\end{eqnarray*}
is also a non-degenerate $d$-dimensional $1$-Meixner random vector.\\
Let us first center the random variables $X_1$, $X_2$, $\dots$, $X_d$, by subtracting for each of them its
expectation. That means, for all $i \in \{1$, $2$, $\dots$, $d\}$, we define:
\begin{eqnarray*}
X_i' & := & X_i - E\left[X_i\right].
\end{eqnarray*}
Thus, $(X_1'$, $X_2'$, $\dots$, $X_d')$ is a non-degenerate $d$-dimensional $1$-Meixner random vector,
in which each component is a random variable with expectation equal to $0$.\\
Now let us apply the Gram-Schmidt orthogonalization procedure to $X_1'$, $X_2'$, $\dots$, $X_d'$, with
respect to the inner product $\langle \cdot$, $\cdot \rangle$ of $L^2(\Omega$, ${\mathcal F}$, $P)$.
We obtain an orthonormal set of random variables $X_1''$, $X_2''$, $\dots$, $X_d''$, which are also the components of a centered non-degenerate $d$-dimensional $1$-Meixner random vector, since
the Gram-Schmidt orthogonalization procedure is obtained via an invertible linear map.\\
Thus, via an invertible affine map, we may assume that $(X_1$, $X_2$, $\dots$, $X_d)$ is a non-degenerate
$d$-dimensional $1$-Meixner random vector such that for all $i$, $j$, and $k$ in
$\{1$, $2$, $\dots$, $d\}$, we have:
\begin{eqnarray*}
E\left[X_iX_j\right] & = & \delta_{i, j}
\end{eqnarray*}
and
\begin{eqnarray*}
E\left[X_k\right] & = & 0,
\end{eqnarray*}
where $\delta_{i, j}$ denotes Kronecker's symbol.\\
Since, for all $1 \leq i \leq d$, we have $E[X_i] = 0$, if we define $\phi := 1$,
i.e. $\phi$ is the constant polynomial equal to $1$, then we have:
\begin{eqnarray*}
U_i\phi & = & a^-(i)\phi + \frac{1}{2}a^0(i)\phi \nonumber\\
& = & 0 + \frac{1}{2}E\left[X_i\right]\phi \nonumber\\
& = & 0,
\end{eqnarray*}
since $\phi \in G_0$, and $a^-(i): G_0 \to G_{-1} = \{0\}$, while:
\begin{eqnarray*}
a^0(i)\phi & = & P_0\left(X_i \cdot 1\right)\\
& = & \left\langle X_i\phi, \phi\right\rangle\phi\\
& = & E\left[X_i\right]\phi,
\end{eqnarray*}
where $P_n$ denotes the orthogonal projection of $L^2(\Omega$, ${\mathcal F}$, $P)$ onto
$G_n$, for all $n \geq 0$.\\
For all $1 \leq i$, $j \leq d$, using the polynomially duality
between $U_i$ and $V_i$, we have:
\begin{eqnarray*}
\delta_{i, j} & = & E\left[X_iX_j\right] \nonumber\\
& = & \left\langle X_iX_j\phi, \phi\right\rangle \nonumber\\
& = & \left\langle \left(U_i + V_i\right)X_j\phi, \phi \right\rangle \nonumber\\
& = & \left\langle U_iX_j\phi, \phi\right\rangle + \left\langle X_j\phi, U_i\phi\right\rangle
\nonumber\\
& = & \left\langle X_jU_i\phi, \phi\right\rangle +
\left\langle \left[U_i, X_j\right]\phi, \phi\right\rangle +  0  \nonumber\\
& = & 0 + \left\langle \left(\sum_{k = 1}^d\alpha_{i, j, k}X_k + \beta_{i, j}I\right)\phi,
\phi\right\rangle \nonumber\\
& = & \sum_{k = 1}^d\alpha_{i, j, k}E\left[X_k\right] + \beta_{i, j}\nonumber\\
& = & \beta_{i, j}. \nonumber
\end{eqnarray*}
Thus, for all $(i$, $j) \in \{1$, $2$, $\dots$, $d\}^2$, we have:
\begin{eqnarray*}
\beta_{i, j} & = & \delta_{i, j}.
\end{eqnarray*}
We establish now two sets of consistency conditions:
\begin{proposition} {\bf Linear Consistency Conditions} \
\label{P:LCC}
For all $(i$, $j$, $k) \in \{1$, $2$, $\dots$, $d\}^3$, we have:
\begin{enumerate}
\item $\alpha_{i, j, k} = \alpha_{j, i, k}$.

\item $\alpha_{i, j, k} = \alpha_{i, k, j}$.

\item for every permutation $\pi$ of the indices $\lp{i, j, k}$:
\begin{eqnarray*}
\alpha_{\pi(i), \pi(j), \pi(k)} & = & \alpha_{i, j, k}.
\end{eqnarray*}
\end{enumerate}
\end{proposition}
\begin{proof}
1. For all $1 \leq i$, $j \leq d$, due to the axiom of Commutative Probability:
\begin{eqnarray*}
\left[U_i, X_j\right] & = & \left[U_j, X_i\right],
\end{eqnarray*}
we have:
\begin{eqnarray*}
\sum_{k = 1}^d\alpha_{i, j, k}X_k + \beta_{i, j}I & = &
\sum_{k = 1}^d\alpha_{j, i, k}X_k + \beta_{j, i}I.
\end{eqnarray*}
Since $X_1$, $X_2$, $\dots$, $X_d$, and $I$ are linearly independent, we conclude that
for all $1 \leq k \leq d$, we have:
\begin{eqnarray*}
\alpha_{i, j, k} & = & \alpha_{j, i, k}.
\end{eqnarray*}
2. For all $1 \leq i$, $j$, $k \leq d$, we can compute the joint moment
$E[X_iX_jX_k]$ in two different ways.\\
Indeed, we have:
\begin{eqnarray*}
E\left[X_iX_jX_k\right] & = & \left\langle \left(U_i + V_i\right)X_jX_k\phi, \phi\right\rangle\\
& = & \left\langle U_iX_jX_k\phi, \phi\right\rangle + \left\langle X_jX_k\phi, U_i\phi\right\rangle\\
& = & \left\langle X_jX_kU_i\phi, \phi\right\rangle +
\left\langle \left[U_i, X_jX_k\right]\phi, \phi\right\rangle + 0.
\end{eqnarray*}
Using now Leibniz commutator rule:
\begin{eqnarray*}
\left[U_i, X_jX_k\right] & = & \left[U_i, X_j\right]X_k  + X_j\left[U_i, X_k\right],
\end{eqnarray*}
we obtain:
\begin{eqnarray}
& \ & E\left[X_iX_jX_k\right] \nonumber\\
& = &
\left\langle \left(\sum_{l = 1}^d \alpha_{i, j, l}X_l + \delta_{i, j}I\right)X_k\phi, \phi\right\rangle
+ \left\langle X_j\left(\sum_{l = 1}^d \alpha_{i, k, l}X_l + \delta_{i, k}I\right)\phi, \phi\right\rangle
\nonumber\\
& = & \sum_{l = 1}^d\alpha_{i, j, l}E\left[X_lX_k\right] +
\sum_{l = 1}^d\alpha_{i, k, l}E\left[X_jX_l\right] \nonumber\\
& = & \sum_{l = 1}^d\alpha_{i, j, l}\delta_{l, k} + \sum_{l = 1}^d\alpha_{i, k, l}\delta_{j, l} \nonumber\\
& = & \alpha_{i, j, k} + \alpha_{i, k, j}. \label{ijk1}
\end{eqnarray}
Permuting now the factors $X_i$, $X_j$, and $X_k$ inside the expectation, a similar computation shows that:
\begin{eqnarray}
E\left[X_iX_jX_k\right] & = & E\left[X_kX_iX_j\right]\nonumber\\
& = & \alpha_{k, i, j} + \alpha_{k, j, i}. \label{ijk2}
\end{eqnarray}
Thus, from (\ref{ijk1}) and (\ref{ijk2}), it follows that:
\begin{eqnarray*}
\alpha_{i, j, k} + \alpha_{i, k, j} & = & \alpha_{k, i, j} + \alpha_{k, j, i},
\end{eqnarray*}
which combined with the fact, from part 1., that:
\begin{eqnarray*}
\alpha_{i, k, j} & = & \alpha_{k, i, j},
\end{eqnarray*}
implies:
\begin{eqnarray*}
\alpha_{i, j, k} & = & \alpha_{k, j, i}.
\end{eqnarray*}
3. Let $i$, $j$, and $k$ be fixed in $\{1$, $2$, $\dots$, $d\}$. Parts 1. and 2, imply that
for all transpositions $\tau$ of $\lp{i, j, k}$, we have:
\begin{eqnarray*}
\alpha_{\tau(i), \tau(j), \tau(k)} & = & \alpha_{i, j, k}.
\end{eqnarray*}
Since every permutation can be written as a product of transpositions, we conclude that
for all permutations $\pi$ of $\lp{i, j, k}$, we have:
\begin{eqnarray*}
\alpha_{\pi(i), \pi(j), \pi(k)} & = & \alpha_{i, j, k}.
\end{eqnarray*}
\end{proof}
\noindent As a consequence of formula (\ref{ijk1}) we obtain:
\begin{corollary}
For all $i$, $j$, and $k$ in $\{1$, $2$, $\dots$, $d\}$, we have:
\begin{eqnarray}
E\left[X_iX_jX_k\right] & = & 2\alpha_{i, j, k}. \label{third_mixed_moment}
\end{eqnarray}
\end{corollary}

\section{Moments estimates and Laplace transform}
\label{S:meLt}

Let $(X_1$, $X_2$, $\dots$, $X_d)$ be a $d$-dimensional random vector.\\
Let $\vct{i} = (i_1$, $i_2$, $\dots$, $i_d) \in [{\mathbb N} \cup \{0\}]^d$.
We introduce the following notations:
\begin{itemize}

\item $|\vct{i}| : = i_1 + i_2 + \cdots + i_d$ and call $|\vct{i}|$ the {\em length} of $\vct{i}$.

\item $X^{\vct{i}} : = X_1^{i_1}X_2^{i_2} \cdots X_d^{i_d}$.

\end{itemize}

We have the following lemma.

\begin{lemma}\label{moment_estimate_lemma}

Let $(X_1$, $X_2$, $\dots$, $X_d)$ be a centered $d$-dimensional 1-Meixner random vector.
Let $\{\alpha_{i, j, k}\}_{1 \leq i, j, k \leq d}$ and $\{\beta_{i, j}\}_{1 \leq i, j \leq d}$ be the coefficients that
are used to express the commutators of the joint semi-annihilation operators and $X_1$, $X_2$, $\dots$, $X_d$ as
linear combinations
of $X_1$, $X_2$, $\dots$, $X_d$, and the identity operator $I$. Then for all
$\vct{i} = (i_1$, $i_2$, $\dots$, $i_d) \in [{\mathbb N} \cup \{0\}]^d$, we have:
\begin{eqnarray}
\left|E\left[X^{\vct{i}}\right]\right| & \leq & K^{|\vct{i}|} \cdot |\vct{i}|!, \label{est_1}
\end{eqnarray}
where $K := \max\{dA + B$, $1\}$, for $A := \max\{|\alpha_{i, j, k}| \mid 1 \leq i$, $j$, $k \leq d\}$,
$B := \max\{|\beta(i, j)| \mid 1 \leq i$, $j \leq d\}$, and
\begin{eqnarray}
E\left[\left|X^{\vct{i}}\right|\right] & \leq & \left(2K\right)^{|\vct{i}|} \cdot |\vct{i}|!. \label{est_2}
\end{eqnarray}

\end{lemma}

\begin{proof}

We will prove first (\ref{est_1}) by induction on $l := |\vct{i}|$.\\
For $l := 0$, the inequality is obvious since:
\begin{eqnarray*}
\left|E\left[X^{\vct{0}}\right]\right| & = & 1\nonumber\\
& \leq & K.
\end{eqnarray*}
Let us assume that inequality (\ref{est_2}) is true for all multi-indexes $\vct{i} \in [{\mathbb N} \cup \{0\}]^d$ of length
$|\vct{i}| \leq l$, and prove that it remains true for all multi-indexes of length $l + 1$.\\
Let $\vct{i} = (i_1$, $i_2$, $\dots$, $i_d) \in [{\mathbb N} \cup \{0\}]^d$ be a multi-index of length
$i_1 + i_2 + \cdots + i_d = l + 1$. Since $l + 1 \geq 1$, there exists $w \in \{1$, $2$, $\dots$, $d\}$, such that
$i_w \geq 1$. Thus, the factor $X_w$ appears for sure in the product $X^{\vct{i}} = X_1^{i_1}X_2^{i_2} \cdots X_d^{i_d}$.\\
We have:
\begin{eqnarray*}
E\left[X^{\vct{i}}\right] & = & \langle X^{\vct{i}}1, 1 \rangle \nonumber\\
& = & \langle X_wX_1^{i_1} \cdots X_w^{i_w - 1} \cdots X_d^{i_d}1, 1 \rangle \nonumber\\
& = & \langle \left(U_w + V_w\right)X_1^{i_1} \cdots X_w^{i_w - 1} \cdots X_d^{i_d}1, 1 \rangle \nonumber\\
& = & \langle U_wX_1^{i_1} \cdots X_w^{i_w - 1} \cdots X_d^{i_d}1, 1 \rangle +
\langle X_1^{i_1} \cdots X_w^{i_w - 1} \cdots X_d^{i_d}1, U_w1 \rangle \nonumber\\
& = & \langle U_wX_1^{i_1} \cdots X_w^{i_w - 1} \cdots X_d^{i_d}1, 1 \rangle,
\end{eqnarray*}
since $U_w1 = 0$ due to the fact that $E[X_w] = 0$ (since $X_w$ is assumed to be centered).\\
Let us define now the vector $\vct{j}$, of length $|\vct{j}| = |\vct{i}| - 1$, by $\vct{j} := (j_1$, $j_2$, $\dots$, $j_d)$, where:
\begin{eqnarray*}
j_r & := & \left\{
\begin{array}{ccc}
i_r & {\rm if} & r \neq w\\
i_w - 1 & {\rm if} & r = w.
\end{array}
\right.
\end{eqnarray*}
We commute $U_w$ with $X^{\vct{j}}$ using Leibniz commutator rule, and obtain:
\begin{eqnarray*}
E\left[X^{\vct{i}}\right] & = & \langle U_wX_1^{j_1} \cdots X_w^{j_w} \cdots X_d^{j_d}1, 1 \rangle \nonumber\\
& = & \langle X_1^{j_1} \cdots X_w^{j_w} \cdots X_d^{j_d}U_w1, 1 \rangle +
\left\langle \left[U_w, X_1^{j_1} \cdots X_w^{j_w} \cdots X_d^{j_d}\right]1, 1 \right\rangle \nonumber\\
& = & \sum_{p = 1}^d\sum_{q = 1}^{j_p}
\left\langle X_1^{j_1} \cdots X_{p - 1}^{j_{p - 1}}X_p^{q - 1}\left[U_w, X_p\right]X_p^{j_p - q}X_{p + 1}^{j_{p + 1}}
\cdots X_d^{j_d}1, 1 \right\rangle
\end{eqnarray*}
since $U_w1 = 0$.
Because:
\begin{eqnarray*}
\left[U_w, X_p\right] & = & \sum_{r = 1}^d\alpha_{w, p, r}X_r + \beta_{w, p}I,
\end{eqnarray*}
we obtain:
\begin{eqnarray}
E\left[X^{\vct{i}}\right]
& = & \sum_{p = 1}^d\sum_{q = 1}^{j_p}
\left\langle X_1^{j_1} \cdots X_p^{q - 1}\left(\sum_{r = 1}^d\alpha_{w, p, r}X_r + \beta_{w, p}I\right)
X_p^{j_p - q}\cdots X_d^{j_d}1, 1 \right\rangle \nonumber\\
& = & \sum_{p = 1}^d\sum_{q = 1}^{j_p}\sum_{r = 1}^d\alpha_{w, p, r}E\left[X_1^{j_1} \cdots X_p^{q - 1}X_rX_p^{j_p - q} \cdots X_d^{j_d}\right]
\nonumber\\
& \ & + \sum_{p = 1}^d\sum_{q = 1}^{j_p}\beta_{w, p}E\left[X_1^{j_1} \cdots X_p^{q - 1}X_p^{j_p - q} \cdots X_d^{j_d}\right]
\nonumber\\
& = & \sum_{p = 1}^d\sum_{r = 1}^dj_p\alpha_{w, p, r}E\left[X_rX_1^{j_1} \cdots X_p^{j_p - 1} \cdots X_d^{j_d}\right] \nonumber\\
& \ & + \sum_{p = 1}^dj_p\beta_{w, p}E\left[X_1^{j_1} \cdots X_p^{j_p - 1} \cdots X_d^{j_d}\right]. \label{temp_1}
\end{eqnarray}
Since $X_rX_1^{j_1} \cdots X_p^{j_p - 1} \cdots X_d^{j_d} = X^{\vct{u}}$, for some vector $\vct{u}$, with $|\vct{u}| = l$,
and\\
$X_1^{j_1} \cdots X_p^{j_p - 1} \cdots X_d^{j_d} = X^{\vct{v}}$, for some vector $\vct{v}$, of length $|\vct{v}| = l - 1$,
using the triangle inequality, the induction hypothesis, and the inequalities $(l - 1)! \leq l!$ and $K^{l - 1} \le K^l$,
we conclude from (\ref{temp_1}) that:
\begin{eqnarray*}
E\left[X^{\vct{i}}\right]
& \leq & \sum_{p = 1}^d\sum_{r = 1}^dj_p \cdot A \cdot K^l \cdot l! + \sum_{p = 1}^dj_p \cdot B \cdot K^{l - 1} \cdot (l - 1)! \nonumber\\
& \leq & dAK^l \cdot l!\left(\sum_{p = 1}^dj_p\right) + BK^l \cdot l!\left(\sum_{p = 1}^dj_p\right) \nonumber\\
& \le & dAK^l \cdot l! \cdot (l + 1) + BK^l \cdot l! \cdot (l + 1) \nonumber\\
& = & K^l(dA + B) \cdot (l + 1)! \nonumber\\
& \leq & K^{l + 1} \cdot (l + 1)!.
\end{eqnarray*}
The proof of part a) is now complete.\\
To prove part b), we use Jensen inequality for the convex function $\varphi(t) = t^2$, and the inequality from part a).
Thus, for all $\vct{i} \in [{\mathbb N} \cup \{0\}]^d$, we have:
\begin{eqnarray}
\left(E\left[\left|X^{\vct{i}}\right|\right]\right)^2 & \leq & E\left[\left|X^{\vct{i}}\right|^2\right] \nonumber\\
& = &  E\left[X^{2\vct{i}}\right] \nonumber\\
& \leq & K^{|2\vct{i}|} \cdot |2\vct{i}|! \nonumber\\
& \leq & K^{2|\vct{i}|} \cdot 2^{2|\vct{i}|}\left(|\vct{i}|!\right)^2, \label{temp_2}
\end{eqnarray}
due to the fact that for all $l \in {\mathbb N} \cup \{0\}$, we have:
\begin{eqnarray*}
\frac{(2l)!}{(l!)^2} & = & {2l \choose l} \nonumber\\
& \leq & \sum_{j = 0}^{2l}{2l \choose j} \nonumber\\
& = & 2^{2l}.
\end{eqnarray*}
Taking the square root in both sides of (\ref{temp_2}) we obtain inequality (\ref{est_2}).
\end{proof}
Because of the above estimates, we have the following:
\begin{lemma}
If $X = (X_1$, $X_2$, $\dots$, $X_d)$ is a $d$-dimensional $1$-Meixner random vector, then the Laplace transform of $X$,
\begin{eqnarray*}
\varphi\left(\vct{t}\right) & = & E\left[\exp\left(\vct{t} \cdot X\right)\right]
\end{eqnarray*}
is well defined and twice differentiable, with continuous second order partial derivatives on a neighborhood
$V$ of $\vct{0} = (0$, $0$, $\dots$, $0) \in {\mathbb R}^d$.
\end{lemma}
\begin{proof}
Let $V := \{\vct{t} \in {\mathbb R}^d \mid \parallel \vct{t}\parallel_{\infty} < R\}$,
where for all $\vct{t} = (t_1$, $t_2$, $\dots$, $t_d) \in {\mathbb R}^d$, we define
$\parallel \vct{t} \parallel_{\infty} := \max\{|t_1|$, $|t_2|$, $\dots$, $|t_d|\}$, and
$R := 1/(2Kd)$, where $K$ is the constant from the previous lemma.\\
For all $n \in {\mathbb N}$, we define
${\mathcal P}_n := \{\sigma : \{1$, $2$, $\dots$, $n\} \to \{1$, $2$, $\dots$, $d\}\}$.
Then for all $\vct{t} \in V$, we have:
\begin{eqnarray*}
\sum_{n = 0}^{\infty}\frac{E[|\vct{t} \cdot X|^n]}{n!} & \leq & 1 +
\sum_{n = 1}^{\infty}\sum_{\sigma \in {\mathcal P}_n}\frac{E\left[|t_{\sigma(1)}||X_{\sigma(1)}||t_{\sigma(2)}||X_{\sigma(2}|
\cdots |t_{\sigma(n)}||X_{\sigma(n)}|\right]}{n!} \nonumber\\
& \leq & 1 + \sum_{n = 1}^{\infty}\sum_{\sigma \in {\mathcal P}_n}
\parallel \vct{t} \parallel_{\infty}^n\frac{E\left[|X_{\sigma(1)}||X_{\sigma(2)}| \cdots |X_{\sigma(n)}|\right]}{n!} \nonumber\\
& \leq & 1 + \sum_{n = 1}^{\infty}\sum_{\sigma \in {\mathcal P}_n}
\parallel t \parallel_{\infty}^n\frac{2^nK^nn!}{n!} \nonumber\\
& = & \sum_{n = 0}^{\infty}2^nK^nd^n\parallel t \parallel_{\infty}^n \nonumber\\
& = & \sum_{n = 0}^{\infty}\left(\frac{\parallel t \parallel_{\infty}}{R}\right)^n \nonumber\\
& = & \frac{R}{R - \parallel t \parallel_{\infty}} \nonumber\\
& < & \infty.
\end{eqnarray*}
Monotone convergence theorem implies now that:
\begin{eqnarray*}
E\left[\exp\left(|\vct{t} \cdot X|\right)\right] & = & E\left[\sum_{n = 0}^{\infty}\frac{|\vct{t} \cdot X|^n}{n!}\right] \nonumber\\
& = & \sum_{n = 0}^{\infty}\frac{E[|\vct{t} \cdot X|^n]}{n!} \nonumber\\
& < & \infty.
\end{eqnarray*}
Thus, we have:
\begin{eqnarray*}
\varphi(t) & = & E\left[\exp\left(\vct{t} \cdot X\right)\right] \nonumber\\
& \leq & E\left[\exp\left(|\vct{t} \cdot X|\right)\right] \nonumber\\
& < & \infty.
\end{eqnarray*}
Therefore, the Laplace transform of $X$, $\varphi$, is well defined on $V$.\\
In the same way, we can see that $\varphi$ is infinitely differentiable on $V$, and each derivative can be performed
term by term using the exponential series.
\end{proof}
We find now a system of partial differential equations for the Laplace transform $\varphi$ of $X$.
We have the following lemma.
\begin{lemma}
\label{L:syst}
Let $X = (X_1$, $X_2$, $\dots$, $X_d)$ be a non-degenerate $d$-dimensional $1$-Meixner random vector.
Let $\{\alpha_{i, j, k}\}_{1 \leq i, j, k \leq d}$ and $\{\beta_{i, j}\}_{1 \leq i, j \leq d}$ be the
coefficients that are used to express
the commutators of their semi-annihilation operators and the components of $X$, as linear combinations of $X_1$, $X_2$, $\dots$, $X_d$.
We also assume that $X_1$, $X_2$, $\dots$, $X_d$ form an orthonormal set of centered random variables in $L^2$ (which we saw before that it is possible to be achieved via a translation and an invertible linear transformation). Thus, for all $1 \leq i$, $j \leq d$, $\beta_{i,j} = \delta_{i, j}$
(the Kronecker symbol). Then the Laplace transform of $X$, which is defined as:
\begin{eqnarray*}
\varphi(t_1, t_2, \cdots, t_d) & := & E\left[\exp\left(t_1X_1 + t_2X_2 + \cdots + t_dX_d\right)\right],
\end{eqnarray*}
for all $\vct{t} = (t_1$, $t_2$, $\dots$, $t_d)$ in a neighborhood $V$ of $\vct{0} = (0$, $0$, $\dots$, $0)$, satisfies the following system of differential equations:
\begin{eqnarray*}
\frac{\partial \varphi}{\partial t_1} & = & \sum_{1 \leq j, k \leq d}\alpha_{1, j, k}t_j\frac{\partial \varphi}{\partial t_k}
+ t_1\varphi\\
\frac{\partial \varphi}{\partial t_2} & = & \sum_{1 \leq j, k \leq d}\alpha_{2, j, k}t_j\frac{\partial \varphi}{\partial t_k}
+ t_2\varphi\\
\vdots & \vdots & \vdots \\
\frac{\partial \varphi}{\partial t_d} & = & \sum_{1 \leq j, k \leq d}\alpha_{d, j, k}t_j\frac{\partial \varphi}{\partial t_1}
+ t_d\varphi.
\end{eqnarray*}
\end{lemma}

\begin{proof}
As we saw in the previous lemma, there exists a neighborhood $V$ of $\vct{0}$, on which the Laplace transform of $\varphi$ is defined and infinitely many times differentiable. Moreover, on that neighborhood the differentiation can be carried out term by term in
the Taylor series of the exponential function, and the differentiation can be interchanged with the expectation.\\
Let $i \in \{1$, $2$, $\dots$, $d\}$ be fixed. For all $\vct{t} := (t_1$, $t_2$, $\dots$, $t_d) \in V$, we have:
\begin{eqnarray*}
& \ & \frac{\partial \varphi}{\partial t_i}(\vct{t}) \nonumber\\
& = & E\left[X_i\exp\left(t_1X_1 + t_2X_2 + \cdots + t_dX_d\right)\right] \nonumber\\
& = & \sum_{n = 0}^{\infty}\frac{1}{n!}E\left[X_i\left(t_1X_1 + t_2X_2 + \cdots + t_dX_d\right)^n\right] \nonumber\\
& = & \sum_{n = 0}^{\infty}\frac{1}{n!}\langle \left(U_i + V_i\right)\left(t_1X_1 + t_2X_2 + \cdots + t_dX_d\right)^n1, 1\rangle
\nonumber\\
& = & \sum_{n = 0}^{\infty}\frac{1}{n!}\langle U_i\left(t_1X_1 + t_2X_2 + \cdots + t_dX_d\right)^n1, 1\rangle \nonumber\\
& \ & +
\sum_{n = 0}^{\infty}\frac{1}{n!}\langle \left(t_1X_1 + t_2X_2 + \cdots + t_dX_d\right)^n1, U_i1\rangle \nonumber\\
& = & \sum_{n = 0}^{\infty}\frac{1}{n!}\langle U_i\left(t_1X_1 + t_2X_2 + \cdots + t_dX_d\right)^n1, 1\rangle,
\end{eqnarray*}
since $U_i1 = (1/2)E[X_i] = 0$.\\
We commute now $U_i$ and $(t_1X_1 + t_2X_2 + \cdots + t_dX_d)^n$, using Leibniz commutation rule, and obtain:
\begin{eqnarray*}
& \ & \frac{\partial \varphi}{\partial t_i}(\vct{t}) \nonumber\\
& = & \sum_{n = 0}^{\infty}\frac{1}{n!}\langle \left(t_1X_1 + t_2X_2 + \cdots + t_dX_d\right)^nU_i1, 1\rangle \nonumber\\
& \ & + \sum_{n = 0}^{\infty}\frac{1}{n!}\langle \left[U_i, \left(t_1X_1 + t_2X_2 + \cdots + t_dX_d\right)^n\right]1, 1\rangle \nonumber\\
\end{eqnarray*}
\begin{eqnarray*}
& = & \sum_{n = 0}^{\infty}\frac{1}{n!}\sum_{p = 1}^n\langle \left(t_1X_1 + t_2X_2 + \cdots + t_dX_d\right)^{p - 1}
\left[U_i, t_1X_1 + t_2X_2 + \cdots + t_dX_d\right] \nonumber\\
& \ & \ \ \ \ \ \ \ \ \ \ \ \ \ \ \ \left(t_1X_1 + t_2X_2 + \cdots + t_dX_d\right)^{n - p}1, 1\rangle \nonumber\\
\end{eqnarray*}
\begin{eqnarray*}
& = & \sum_{n = 0}^{\infty}\frac{1}{n!}\sum_{p = 1}^n\langle \left(t_1X_1 + t_2X_2 + \cdots + t_dX_d\right)^{p - 1}
\sum_{j = 1}^dt_j\left[U_i, X_j\right] \nonumber\\
& \ & \ \ \ \ \ \ \ \ \ \ \ \ \ \ \ \left(t_1X_1 + t_2X_2 + \cdots + t_dX_d\right)^{n - p}1, 1\rangle \nonumber\\
\end{eqnarray*}
\begin{eqnarray*}
& = & \sum_{n = 0}^{\infty}\frac{1}{n!}\sum_{p = 1}^n\langle \left(t_1X_1 + t_2X_2 + \cdots + t_dX_d\right)^{p - 1}
\sum_{j = 1}^dt_j\left(\sum_{k = 1}^d\alpha_{i, j, k}X_k + \delta_{i, j}I\right) \nonumber\\
& \ & \ \ \ \ \ \ \ \ \ \ \ \ \ \ \ \left(t_1X_1 + t_2X_2 + \cdots + t_dX_d\right)^{n - p}1, 1\rangle \nonumber\\
\end{eqnarray*}
\begin{eqnarray*}
& = & \sum_{n = 0}^{\infty}\sum_{p = 1}^n\sum_{j = 1}^d\sum_{k = 1}^d\frac{1}{n!}
\alpha_{i, j, k}t_j\langle X_k\left(t_1X_1 + t_2X_2 + \cdots + t_dX_d\right)^{n - 1}1, 1\rangle \nonumber\\
& \ & + \sum_{n = 0}^{\infty}\sum_{p = 1}^n\frac{1}{n!}
t_i\langle \left(t_1X_1 + t_2X_2 + \cdots + t_dX_d\right)^{n - 1}1, 1\rangle \nonumber\\
\end{eqnarray*}
\begin{eqnarray*}
& = & \sum_{n = 0}^{\infty}\sum_{j = 1}^d\sum_{k = 1}^dn\frac{1}{n!}
\alpha_{i, j, k}t_jE\left[X_k\left(t_1X_1 + t_2X_2 + \cdots + t_dX_d\right)^{n - 1}\right] \nonumber\\
& \ & + t_i\sum_{n = 0}^{\infty}n\frac{1}{n!}
E\left[\left(t_1X_1 + t_2X_2 + \cdots + t_dX_d\right)^{n - 1}\right] \nonumber\\
\end{eqnarray*}
\begin{eqnarray*}
& = & \sum_{j = 1}^d\sum_{k = 1}^d\alpha_{i, j, k}t_j\sum_{n = 1}^{\infty}\frac{1}{(n - 1)!}E\left[X_k\left(t_1X_1 + t_2X_2 + \cdots + t_dX_d\right)^{n - 1}\right] \nonumber\\
& \ & + t_i\sum_{n = 1}^{\infty}\frac{1}{(n - 1)!}E\left[\left(t_1X_1 + t_2X_2 + \cdots + t_dX_d\right)^{n - 1}\right] \nonumber\\
\end{eqnarray*}
\begin{eqnarray*}
& = & \sum_{j = 1}^d\sum_{k = 1}^d\alpha_{i, j, k}t_jE\left[X_k\exp\left(t_1X_1 + t_2X_2 + \cdots + t_dX_d\right)\right] \nonumber\\
& \ & + t_iE\left[\exp\left(t_1X_1 + t_2X_2 + \cdots + t_dX_d\right)\right] \nonumber\\
\end{eqnarray*}
\begin{eqnarray*}
& = & \sum_{j = 1}^d\sum_{k = 1}^d\alpha_{i, j, k}t_j\frac{\partial \varphi}{\partial t_k}\left(\vct{t}\right) + t_i\varphi\left(\vct{t}\right).
\end{eqnarray*}
The proof of this lemma is now complete.
\end{proof}

At this point, it will be desirable to integrate the system in Lemma \ref{L:syst}. If this is achieved, then inverting the Laplace transform
will produce the probability distribution of the 1-Meixner random vector $(X_1$, $X_2$, $\dots$, $X_d)$. While this seems to be a challenging
task, in the following two sections we present results in this direction. In the next section we derive an important necessary condition
for the integrability of the system, for any dimension $d \geq 2$, while in the last section we characterize (describe completely) the
$1$-Meixner random vectors in the case when the dimension is $d = 3$.

\section{Necessary conditions for integrability}
\label{S:nci}
In this section we find a set of  necessary conditions for the integrability of the system from Lemma \ref{L:syst}.
Provided the system has a smooth solution $\varphi$, then $\varphi$ is positive on a neighborhood $V$ of $\vct{0}$, since:
\begin{eqnarray*}
\varphi(\vct{0}) & = & E\left[\exp(\vct{0} \cdot X)\right] \nonumber\\
& = & 1.
\end{eqnarray*}
We introduce the following notations:
we denote by "$\langle \cdot , \cdot \rangle$" the inner product with respect to the original
probability $P$, and by "$\cdot$" the standard inner product in ${\mathbb R}^d$:
\begin{eqnarray*}
\left(x_1', x_2', \cdots, x_d'\right) \cdot \left(x_1'', x_2'', \dots, x_d''\right) & := & x_1'x_1'' + x_2'x_2'' + \cdots + x_d'x_d''.
\end{eqnarray*}
We have the following result:
\begin{lemma}
Let $\{\alpha_{i, j, k}\}_{1 \leq i, j, k \leq d}$ be real numbers, such that for every $(i$, $j$, $k) \in \{1$, $2$, $\dots$, $d\}^3$
and every permutation $\pi$ of $\{i$, $j$, $k\}$, we have:
\begin{eqnarray*}
\alpha_{\pi(i), \pi(j), \pi(k)} & = & \alpha_{i, j, k}.
\end{eqnarray*}
If the system of partial differential equations:
\begin{eqnarray}
\frac{\partial \varphi}{\partial t_1} & = & \sum_{j, k}\alpha_{1, j, k}t_j\frac{\partial \varphi}{\partial t_k} + t_1\varphi \nonumber\\
\frac{\partial \varphi}{\partial t_2} & = & \sum_{j, k}\alpha_{2, j, k}t_j\frac{\partial \varphi}{\partial t_k} + t_2\varphi \label{system}\\
\vdots & \vdots & \vdots \nonumber\\
\frac{\partial \varphi}{\partial t_d} & = & \sum_{j, k}\alpha_{d, j, k}t_j\frac{\partial \varphi}{\partial t_k} + t_d\varphi \nonumber
\end{eqnarray}
has a solution $\varphi$ of class $C^2$ defined on a neighborhood $V$ of ${\bf 0} = (0$, $0$, $\dots$, $0)$, such that $\varphi({\bf 0}) \neq 0$, then for all
$(i$, $j) \in \{1$, $2$, $\dots$, $d\}^2$ and ${\bf t} = (t_1$, $t_2$, $\dots$, $t_d)$ in a neighborhood of ${\bf 0}$, we have:
\begin{eqnarray}
\left(C_{i, j}{\bf t}\right) \cdot \left(\left(I - t_1A_1 - t_2A_2 - \cdots - t_dA_d\right)^{-1}{\bf t}\right) & = & 0, \label{inverse_necessary_condition}
\end{eqnarray}
where, for all $k \in \{1$, $2$, $\dots$, $d\}$, we define the $d \times d$ matrix:
\begin{eqnarray}
A_k & := & \left(\alpha_{k, r, s}\right)_{1 \leq r, s \leq d},
\end{eqnarray}
and
\begin{eqnarray}
C_{i, j} & := & \left[A_i, A_j\right]
\end{eqnarray}
is the commutator of $A_i$ and $A_j$, and $I$ is the $d \times d$ identity matrix.
\end{lemma}
\begin{proof}
For $i = j$, $C_{i, j} = [A_i$, $A_i] = 0$, and so formula (\ref{inverse_necessary_condition}) is obvious. So, we may assume that $i \neq j$.\\
Let us differentiate both sides of the $i$th equation of the system (\ref{system}) with respect to the variable $t_j$. We have:
\begin{eqnarray}
\frac{\partial}{\partial t_j}\frac{\partial \varphi}{\partial t_i} & = & \frac{\partial}{\partial t_j}
\left(\sum_{p, q}\alpha_{i, p, q}t_p\frac{\partial \varphi}{\partial t_q} + t_i\varphi\right).
\end{eqnarray}
This is equivalent to:
\begin{eqnarray}
\frac{\partial^2 \varphi}{\partial t_j\partial t_i} & = &
\sum_q\alpha_{i, j, q}\frac{\partial \varphi}{\partial t_q} +
\sum_{p, q}\alpha_{i, p, q}t_p\frac{\partial^2 \varphi}{\partial t_j\partial t_q} +
t_i\frac{\partial \varphi}{\partial t_j}.
\end{eqnarray}
Applying now the Young's commutation theorem:
\begin{eqnarray}
\frac{\partial^2 \varphi}{\partial t_j\partial t_q} & = & \frac{\partial^2 \varphi}{\partial t_q\partial t_j},
\end{eqnarray}
for all $q \in \{1$, $2$, $\dots$, $d\}$, and using the $j$th equation of the system (\ref{system}), we obtain:
\begin{eqnarray}
\frac{\partial^2 \varphi}{\partial t_j\partial t_i} & = & \sum_q\alpha_{i, j, q}\frac{\partial \varphi}{\partial t_q}
 + \sum_{p, q}\alpha_{i, p, q}t_p\frac{\partial}{\partial t_q}\frac{\partial \varphi}{\partial t_j}
 + t_i\frac{\partial \varphi}{\partial t_j} \nonumber\\
& = & \sum_q\alpha_{i, j, q}\frac{\partial \varphi}{\partial t_q} +
\sum_{p, q}\alpha_{i, p, q}t_p\frac{\partial}{\partial t_q}
\left[\sum_{r, s}\alpha_{j, r, s}t_r\frac{\partial \varphi}{\partial t_s} + t_j\varphi\right]
+ t_i\frac{\partial \varphi}{\partial t_j} \nonumber\\
& = & \sum_q\alpha_{i, j, q}\frac{\partial \varphi}{\partial t_q} + \sum_{p, q, s}\alpha_{i, p, q}\alpha_{j, q, s}t_p\frac{\partial \varphi}{\partial t_s}
+ \sum_{p, q, r, s}\alpha_{i, p, q}\alpha_{j, r, s}t_pt_r
\frac{\partial^2\varphi}{\partial t_q\partial t_s}
 \nonumber\\
& \ &
+ \sum_{p}\alpha_{i, p, j}t_p\varphi +
	\sum_{p, q}\alpha_{i, p, q}t_pt_j\frac{\partial \varphi}{\partial t_q}
	+  t_i\frac{\partial \varphi}{\partial t_j}.
 \label{ij}
\end{eqnarray}
Switching the roles of $i$ and $j$, similarly, we can prove that:
\begin{eqnarray}
\frac{\partial^2 \varphi}{\partial t_i\partial t_j} & = & \sum_q\alpha_{j, i, q}\frac{\partial \varphi}{\partial t_q} + \sum_{p, q, s}\alpha_{j, p, q}\alpha_{i, q, s}t_p
\frac{\partial \varphi}{\partial t_s}
+ \sum_{p, q, r, s}\alpha_{j, p, q}\alpha_{i, r, s}t_pt_r
\frac{\partial^2\varphi}{\partial t_q\partial t_s}
\nonumber\\
& \ &
+ \sum_{p}\alpha_{j, p, i}t_p\varphi+
	\sum_{p, q}\alpha_{j, p, q}t_pt_i\frac{\partial \varphi}{\partial t_q}
	+  t_j\frac{\partial \varphi}{\partial t_i}.
\label{ji}
\end{eqnarray}
Since $\partial^2\varphi/(\partial t_i\partial t_j) = \partial^2\varphi/(\partial t_j\partial t_i)$, formulas (\ref{ij}) and
(\ref{ji}) imply:
\begin{eqnarray*}
& \sum_q\alpha_{i, j, q}\frac{\partial \varphi}{\partial t_q} + \sum_{p, q, s}\alpha_{i, p, q}\alpha_{j, q, s}t_p\frac{\partial \varphi}{\partial t_s}
+ \sum_{p, q, r, s}\alpha_{i, p, q}\alpha_{j, r, s}t_pt_r
\frac{\partial^2\varphi}{\partial t_q\partial t_s}
\nonumber\\
&+ \sum_{p}\alpha_{i, p, j}t_p\varphi +
	\sum_{p, q}\alpha_{i, p, q}t_pt_j\frac{\partial \varphi}{\partial t_q}
	+  t_i\frac{\partial \varphi}{\partial t_j}
\\
&= \sum_q\alpha_{j, i, q}\frac{\partial \varphi}{\partial t_q} + \sum_{p, q, s}\alpha_{j, p, q}\alpha_{i, q, s}t_p
\frac{\partial \varphi}{\partial t_s}
+ \sum_{p, q, r, s}\alpha_{j, p, q}\alpha_{i, r, s}t_pt_r
\frac{\partial^2{\color{blue} \varphi}}{\partial t_q\partial t_s}
\nonumber\\
& + \sum_{p}\alpha_{j, p, i}t_p\varphi+
	\sum_{p, q}\alpha_{j, p, q}t_pt_i\frac{\partial \varphi}{\partial t_q}	+  t_j\frac{\partial \varphi}{\partial t_i}.
\end{eqnarray*}
Since, we have: $\sum_q\alpha_{i, j, q}\frac{\partial \varphi}{\partial t_q} = \sum_q\alpha_{j, i, q}\frac{\partial \varphi}{\partial t_q}$,
$\sum_{p, q, r, s}\alpha_{i, p, q}\alpha_{j, r, s}t_pt_r\frac{\partial^2 \varphi}{\partial t_q\partial t_s} =
\sum_{p, q, r, s}\alpha_{j, p, q}\alpha_{i, r, s}t_pt_r\frac{\partial^2 \varphi}{\partial t_q\partial t_s}$,
and $\sum_{p}\alpha_{i, p, j} t_p\varphi = \sum_{p}\alpha_{j, p, i} t_p\varphi$,
 we conclude from the last formula
that:
\begin{eqnarray}
& &\sum_{p, q, s}\alpha_{i, p, q}\alpha_{j, q, s}t_p\frac{\partial \varphi}{\partial t_s}
	+\sum_{p, q}\alpha_{i, p, q}t_pt_j\frac{\partial \varphi}{\partial t_q}
	+  t_i\frac{\partial \varphi}{\partial t_j} \\
\label{E:mixedp}
&= &
 \sum_{p, q, s}\alpha_{j, p, q}\alpha_{i, q, s}t_p\frac{\partial \varphi}{\partial t_s}
 +\sum_{p, q}\alpha_{j, p, q}t_pt_i\frac{\partial \varphi}{\partial t_q}	+  t_j\frac{\partial \varphi}{\partial t_i}. \nonumber
\end{eqnarray}
 Note that from the system \eqref{system}
 	we have
 	\[ \sum_{p, q}\alpha_{i, p, q}t_pt_j\frac{\partial \varphi}{\partial t_q} =
 	t_j\lp{\frac{\partial \varphi}{\partial t_i}- t_i\vp},
 	\]
 	and similarly
 		\[ \sum_{p, q}\alpha_{j, p, q}t_pt_i\frac{\partial \varphi}{\partial t_q} =
 	t_i\lp{\frac{\partial \varphi}{\partial t_j}- t_j\vp}.
 	\]
 	Therefore, the equalities \eqref{E:mixedp} can be written:
 	\begin{eqnarray*}
 	& &\sum_{p, q, s}\alpha_{i, p, q}\alpha_{j, q, s}t_p\frac{\partial \varphi}{\partial t_s}
 		+t_j\frac{\partial \varphi}{\partial t_i}- t_jt_i\vp
 		+  t_i\frac{\partial \varphi}{\partial t_j} \\
 	&= &
 	\sum_{p, q, s}\alpha_{j, p, q}\alpha_{i, q, s}t_p\frac{\partial \varphi}{\partial t_s}
 + t_i\frac{\partial \varphi}{\partial t_j}- t_it_j\vp
 	+  t_j\frac{\partial \varphi}{\partial t_i}. \nonumber
 	\end{eqnarray*}	
That means:
\begin{eqnarray*}
\sum_{p, s}t_p\left(A_iA_j\right)_{ps}\frac{\partial \varphi}{\partial t_s} & = &
\sum_{p, s}t_p\left(A_jA_i\right)_{ps}\frac{\partial \varphi}{\partial t_s}.
\end{eqnarray*}
This means:
\begin{eqnarray*}
{\bf t} \cdot A_iA_j\nabla\varphi({\bf t}) & = & {\bf t}\cdot A_jA_i\nabla\varphi({\bf t}),
\end{eqnarray*}
which due to the fact that both $A_i$ and $A_j$ are self-adjoint matrices, is equivalent to:
\begin{eqnarray*}
A_jA_i{\bf t} \cdot \nabla\varphi({\bf t}) & = & A_iA_j{\bf t}\cdot \nabla\varphi({\bf t}).
\end{eqnarray*}
This last equation is equivalent to:
\begin{eqnarray}
\left[A_i, A_j\right]{\bf t} \cdot \nabla\varphi({\bf t}) & = & 0,
\label{commutator_nabla_orthogonal1}
\end{eqnarray}
for all ${\bf t}$ in a neighborhood $V$ of ${\bf 0}$.
Since our system of partial differential equations (\ref{system}) is equivalent to:
\begin{eqnarray*}
\nabla\varphi({\bf t}) & = & \left(I - t_1A_1 - t_2A_2 - \cdots - t_dA_d\right)^{-1}\varphi({\bf t}){\bf t},
\end{eqnarray*}
equation (\ref{commutator_nabla_orthogonal1}) is equivalent to:
\begin{eqnarray*}
\left[A_i, A_j\right]{\bf t} \cdot \left(I - t_1A_1 - t_2A_2 - \cdots - t_dA_d\right)^{-1}\varphi({\bf t}){\bf t} & = & 0, \label{commutator_nabla_orthogonal}
\end{eqnarray*}
and since $\varphi({\bf 0}) = 1 \neq 0$, we can divide the equation by $\varphi({\bf t})$, on a neighborhood $V$ of ${\bf 0}$, and conclude that:
\begin{eqnarray*}
\left[A_i, A_j\right]{\bf t} \cdot \left(I - t_1A_1 - t_2A_2 - \cdots - t_dA_d\right)^{-1}{\bf t} & = & 0.
\end{eqnarray*}
Thus, the Lemma is proved.
\end{proof}

\begin{proposition}
The necessary condition (\ref{inverse_necessary_condition}) is equivalent to:
\begin{eqnarray}
C_{i,j}{\bf t} \cdot \left(t_1A_1 + t_2A_2 + \cdots + t_dA_d\right)^n{\bf t} & = & 0,
\label{all_n_natural}
\end{eqnarray}
for all $n \in {\mathbb N}$ and all ${\bf t} = (t_1$, $t_2$, $\dots$, $t_d) \in {\mathbb R}^d$, which is turn is equivalent to:
\begin{eqnarray}
C_{i,j}{\bf t} \cdot \left(t_1A_1 + t_2A_2 + \cdots + t_dA_d\right)^n{\bf t} & = & 0, \label{all_n_less_than_d}
\end{eqnarray}
for all $1 \leq n \leq d - 1$ and all ${\bf t} = (t_1$, $t_2$, $\dots$, $t_d) \in {\mathbb R}^d$.
\end{proposition}

\begin{proof}
Since, for all ${\bf t} = (t_1$, $t_2$, $\dots$, $t_d)$ in a neighborhood $V$ of ${\bf 0} = (0$, $0$, $\dots$. $0)$, we have:
\begin{eqnarray}
\left(I - t_1A_1 - t_2A_2 - \cdots - t_dA_d\right)^{-1} & = & \sum_{n = 0}^{\infty}
\left(t_1A_1 + t_2A_2 + \cdots + t_dA_d\right)^n,
\end{eqnarray}
equation (\ref{inverse_necessary_condition}) becomes:
\begin{eqnarray}
\sum_{n = 0}^{\infty}\left(C_{i, j}{\bf t}\right) \cdot
\left(\left(t_1A_1 + t_2A_2 + \cdots + t_dA_d\right)^n{\bf t}\right) & = & 0.
\end{eqnarray}
Due to the fact that, for all $n \geq 0$, $(C_{i, j}{\bf t}) \cdot (t_1A_1 + t_2A_2 + \cdots + t_dA_d)^n{\bf t}$
is a homogenous polynomial of degree $\lp{n + 2}$ in the variables $t_1$, $t_2$, $\dots$, $t_d$, we conclude that for
all $n \geq 0$, we have:
\begin{eqnarray*}
\left(C_{i, j}{\bf t}\right) \cdot
\left(\left(t_1A_1 + t_2A_2 + \cdots + t_dA_d\right)^n{\bf t}\right) & = & 0,
\end{eqnarray*}
for ${\bf t}$ not only in a neighborhood $V$ of ${\bf 0}$ but in the whole space ${\mathbb R}^d$.\\
For a fixed ${\bf t} = (t_1$, $t_2$, $\dots$, $t_d) \in {\mathbb R}^d$, using Cayley-Hamilton-Frobenius Theorem,
the matrix $A_{{\bf t}} := t_1A_1 + t_2A_2 + \cdots + t_dA_d$
satisfies  its own characteristic equation:
\begin{eqnarray}
\det\left(xI - A_{\bf t}\right) & = & 0,
\end{eqnarray}
which is a polynomial equation of degree $d$:
\begin{eqnarray}
x^d + c_{d - 1}x^{d - 1} + \cdots + c_1x + c_0 & = & 0,
\end{eqnarray}
for some real numbers $c_0$, $c_1$, $\dots$, $c_{d - 1}$.\\
It follows from here that each of the matrices: $A_{{\bf t}}^d$, $A_{{\bf t}}^{d + 1}$, $A_{{\bf t}}^{d + 2}$, $\dots$
is a linear combination of $I$, $A_{{\bf t}}$, $A_{{\bf t}}^2$, $\dots$, $A_{{\bf t}}^{d - 1}$. Thus, the condition:
for all $n \geq 0$, we have
\begin{eqnarray*}
\left(C_{i, j}{\bf t}\right) \cdot
\left(\left(t_1A_1 + t_2A_2 + \cdots + t_dA_d\right)^n{\bf t}\right) & = & 0,
\end{eqnarray*}
is equivalent to: for all $1 \leq n \leq d - 1$,
\begin{eqnarray*}
\left(C_{i, j}{\bf t}\right) \cdot
\left(\left(t_1A_1 + t_2A_2 + \cdots + t_dA_d\right)^n{\bf t}\right) & = & 0.
\end{eqnarray*}
Note that for $n = 0$, due to the fact that $C_{i, j} = [A_i$, $A_j]$ is skew-symmetric,
we have:
\begin{eqnarray*}
\left(C_{i, j}{\bf t}\right) \cdot {\bf t} & = & 0.
\end{eqnarray*}
\end{proof}
For $n = 1$, equation ({\ref{all_n_natural}}) becomes:
\begin{eqnarray}
\left(C_{i, j}{\bf t}\right) \cdot \left(\left(t_1A_1 + t_2A_2 + \cdots + t_dA_d\right){\bf t}\right) & = & 0. \label{n=1}
\end{eqnarray}
\begin{proposition}\label{P:Fconst}
For any $\xi \in {\mathbb R}^d$, the cubic homogenous polynomial:
\begin{eqnarray}
F\left({\bf t}\right) & := & \sum_{i, j, k}\alpha_{i, j, k}t_it_jt_k
\label{E:F}
\end{eqnarray}
is constant along the points of the curve:
\begin{eqnarray*}
{\bf t}(s, \xi) & := & \exp\left(sC_{i, j}\right)\xi,
\end{eqnarray*}
which satisfies the initial value problem:
\begin{eqnarray}
\left\{
\begin{array}{ccc}
\frac{d}{ds}{\bf t}(s, \xi) & = & C_{i, j}{\bf t}(s, \xi)\\
\ & \ & \ \\
{\bf t}(0, \xi) & = & \xi.
\end{array}
\right. \label{E:circ}
\end{eqnarray}
\end{proposition}
\begin{proof}
Indeed, for any $i \in \{1$, $2$, $\dots$, $d\}$, we have (Euler's formula):
\begin{eqnarray}
\frac{\partial F}{\partial t_i}({\bf t}) & = & \sum_{p,q,r}\alpha_{p,q,r}\frac{\partial}{\partial t_i}\left(t_pt_qt_r\right) \nonumber\\
& = & \sum_{p,q,r}\alpha_{p,q,r}\left(\delta_{ip}t_qt_r + t_p\delta_{iq}t_r + t_pt_q\delta_{ir}\right) \nonumber\\
& = & \sum_{q,r}\alpha_{i,q,r}t_qt_r + \sum_{p,r}\alpha_{p,i,r}t_pt_r + \sum_{p,q}\alpha_{p,q,i}t_pt_q \nonumber\\
& = & 3\sum_{j, k}\alpha_{i,j,k}t_jt_k,
\end{eqnarray}
since $\alpha_{u,v,w} = \alpha_{\pi(u),\pi(v),\pi(w)}$, for all $(u$, $v$, $w) \in \{1$, $2$, $\dots$, $d\}^3$ and any
$\pi$ permutation of the triplet $(u,v,w)$.
It follows from here that:
\begin{eqnarray}
 \frac{d}{ds}F\left({\bf t}(s, \xi)\right)
& = & \nabla F\left({\bf t(s, \xi)}\right) \cdot \frac{d}{ds}{\bf t}(s, \xi) \nonumber\\
& = & \nabla F\left({\bf t(s, \xi)}\right) \cdot C_{i, j}{\bf t}(s, \xi) \nonumber\\
& = & \left(3\sum_{j, k}\alpha_{1,j,k}t_jt_k, 3\sum_{j, k}\alpha_{2,j,k}t_jt_k, \dots, 3\sum_{j, k}\alpha_{d,j,k}t_jt_k,\right)
\cdot C_{i, j}{\bf t}(s, \xi) \nonumber\\
& = & 3\left(\sum_{j, k}t_j\alpha_{j,1,k}t_k, \sum_{j, k}t_j\alpha_{j,2,k}t_k, \dots, \sum_{j, k}t_j\alpha_{j,d,k}t_k,\right)
\cdot C_{i, j}{\bf t}(s, \xi) \nonumber\\
& = & 3\left(\sum_{j = 1}^d\left(t_jA_j{\bf t}\right)_{1}, \sum_{j = 1}^d\left(t_jA_j{\bf t}\right)_{2}, \dots,
\sum_{j = 1}^d\left(t_jA_j{\bf t}\right)_{d}\right) \cdot C_{i, j}{\bf t}(s, \xi) \nonumber\\
& = & 3\left(\left(\sum_{j = 1}^dt_jA_j\right){\bf t}\right) \cdot C_{i, j}{\bf t}(s, \xi)\nonumber\\
& = & 0,
\end{eqnarray}
by equation (\ref{n=1}).
\end{proof}

\begin{proposition}
	\label{P:3moment}
For all $\vct{t} = (t_1$, $t_2$, $\dots$, $t_d) \in {\mathbb R}^d$, we have:
\[
F\left(\vct{t}\right)  =  \frac{1}{2}E\left[\left(\vct{t} \cdot X\right)^3\right],
\]
where $X := (X_1$, $X_2$, $\dots$, $X_d)$.
\end{proposition}
\proof
Indeed, using formula (\ref{third_mixed_moment}), we have:
\begin{eqnarray*}
F\left(\vct{t}\right) & = & \sum_{i, j, k}\alpha_{i, j, k}t_it_jt_k \nonumber\\
& = & \sum_{i, j, k}\frac{1}{2}E\left[X_iX_jX_k\right]t_it_jt_k \nonumber\\
& = & \frac{1}{2}\sum_{i, j, k}E\left[t_iX_it_jX_jt_kX_k\right] \nonumber\\
& = & \frac{1}{2}E\left[\left(\sum_it_iX_i\right)\left(\sum_jt_jX_j\right)\left(\sum_kt_kX_k\right)\right] \nonumber\\
& = & \frac{1}{2}E\left[\left(\vct{t} \cdot X\right)^3\right].
\end{eqnarray*}
\endproof

\section{The case $d=3$}
\label{S:de3}

In this section we restrict  our attention to the case $d = 3$.
	We give a complete description of
	the non-degenerate 3-dimensional 1-Meixner random vectors, which is contained in Theorem~\ref{T:main},
	stated at the end of the section.\\
	It turns out that unlike the case $d=2$, in $\reo^3$ there exist random vectors whose components
	{\em are not} independent one-dimensional random variables.\\
\ \\
We distinguish between two cases:\\
\ \\
	{\bf Case I.} At least one of the commutators $[A_1,A_2]$, $[A_2,A_3]$, or $[A_1,A_3]$ is non-zero.
\begin{lemma}
	\label{L:Fcanonical}
	If any one of the commutators $[A_1,A_2]$, $[A_2,A_3]$, or $[A_1,A_3]$, is non-zero,
	 then there exists an orthogonal matrix $U$ such that
	 the cubic form:
	\[F_X\left({\bf t}\right)  :=  \sum_{i, j, k}\alpha_{i, j, k}t_it_jt_k\]
is written in the canonical form
	\[F_X\left(U^{-1}{\bf s}\right)= 	F_{UX} \left({\bf s}\right) =
	3as_3(s_1^2+s_2^2)+as_3^3,\]
	for some $a\neq 0$.
\end{lemma}
\proof
  Without loss of generality we may assume that $C=C_{1,2}:= [A_1,A_2] \neq 0$.
Since $C$ is a skew--symmetric real matrix, it must have a non--zero
purely imaginary eigenvalue $i\lambda$.
Then there exists an orthonormal basis $\{f_1$, $f_2$, $f_3\}$ of ${\mathbb R^3}$, such that:
\begin{eqnarray*}
C f_1 & = & \lambda f_2\\
C f_2 & = & -\lambda f_1\\
C f_3 & = & 0.
\end{eqnarray*}
Let $\vct{\xi} := xf_1 + yf_2 + zf_3\in {\mathbb R}^3$, where $x$, $y$, and $z$ are fixed real numbers.\\
In the basis $\{f_1$, $f_2$, $f_3\}$, the curve (level curve for $F_X$) described in Proposition~\ref{P:Fconst} is:
\begin{eqnarray*}
\vct{t}(s,\vct{\xi})& = & \exp\left(sC \right)\vct{\xi} \nonumber\\
& = & \sum_{n = 0}^{\infty}\frac{s^n}{n!}C^n\left(xf_1 + yf_2 + zf_3\right) \nonumber\\
& = & zf_3 + \sum_{n = 0}^{\infty}\frac{s^n}{n!}C^n\left(xf_1 + yf_2\right) \nonumber\\
& = & zf_3 + x\sum_{n = 0}^{\infty}\frac{s^{2n}}{(2n)!}(-1)^n\lambda^{2n}f_1
+ y\sum_{n = 0}^{\infty}\frac{s^{2n}}{(2n)!}(-1)^n\lambda^{2n}f_2
\nonumber\\
& \ & + x\sum_{n = 0}^{\infty}\frac{s^{2n + 1}}{(2n + 1)!}(-1)^n\lambda^{2n + 1}f_2 -
y\sum_{n = 0}^{\infty}\frac{s^{2n + 1}}{(2n + 1)!}(-1)^n\lambda^{2n + 1}f_1 \nonumber\\
& = & zf_3 + \left[x\cos(\lambda s) - y\sin(\lambda s)\right]f_1 + \left[x\sin(\lambda s) + y\cos(\lambda s)\right]f_2.
\end{eqnarray*}
The right hand side in the last formula represents the parametric equation of a circle centered at $(0$, $0$, $z)$, that sits in a plane
perpendicular to $f_3$ and has radius $r = \sqrt{x^2 + y^2}$.
Therefore, by Proposition~\ref{P:Fconst}, $F_X$ is constant on all circles that are centered at a point found on the
${\mathbb R}f_3$--axis and sit in a plane perpendicular to $f_3$.\\
Now, let us consider the orthogonal transformation $U$ that maps the basis
$f_1$, $f_2$, and $f_3$ into the standard basis
$e_1 = (1$, $0$, $0)$, $e_2 = (0$, $1$, $0)$, and $e_3 = (0$, $0$, $1)$, where
$f_1$, $f_2$, and $f_3$ is a basis of $\reo^3$, such that $F$ is rotationally invariant about $f_3$. Therefore,
\[Uf_1 = e_1, \ \ \ Uf_2 = e_2, \ \ \ Uf_3 = e_3.\]
Since $U$ is an orthogonal transformation, it preserves the standard inner product in $\reo^3$. By Proposition~\ref{P:3moment}
we have:
\begin{eqnarray*}
F_{X}(\vct{t}) & = & \frac{1}{2}E\left[\left(\vct{t} \cdot X \right)^3\right] \nonumber\\
& = & \frac{1}{2}E\left[\left(U\vct{t} \cdot UX \right)^3\right] \nonumber\\
& = & F_{UX}(U\vct{t})\\
& = & \sum_{i, j, k}\beta_{i, j, k}s_is_js_k,
\end{eqnarray*}
where:
\begin{eqnarray*}
\vct{s} & = & (s_1, s_2, s_3) \nonumber\\
& := & U\vct{t}
\end{eqnarray*}
and $\beta_{i, j, k}$ are the numbers $\alpha_{i, j, k}$ that correspond to the new non-degenerate $1$-Meixner
random vector:
\begin{eqnarray*}
X' & := & UX.
\end{eqnarray*}
Let us consider a circle of the form:
\begin{eqnarray}
{\mathcal C}_e(c, r) & := & \{s_1e_1 + s_2e_2 + s_3e_3 \mid s_1^2 + s_2^2 = r^2, s_3 = c\},
\end{eqnarray}
for some fixed numbers $c$ and $r$.
For any ${\bf s} \in {\mathcal C}_e(c, r)$, we have:
\begin{eqnarray}
F_{X'}({\bf s}) & = & F_{UX}\left(U\left(U^{-1}{\bf s}\right)\right) \nonumber\\
& = & F_X\left(U^{-1}{\bf s}\right) \nonumber\\
& = & F_X\left(U^{-1}\left(s_1e_1 + s_2e_2 + s_3e_3\right)\right) \nonumber\\
& = & F_X\left(s_1U^{-1}e_1 + s_2U^{-1}e_2 + s_3U^{-1}e_3\right) \nonumber\\
& = & F_X\left(s_1f_1 + s_2f_2 + s_3f_3\right) \nonumber\\
& = & {\rm constant},
\end{eqnarray}
since $F_X$ is constant along circles of the form:
\begin{eqnarray}
{\mathcal C}_f(c, r) & := & \{s_1f_1 + s_2f_2 + s_3f_3 \mid s_1^2 + s_2^2 = r^2, s_3 = c\}.
\end{eqnarray}
Thus $F_{X'} = F_{UX}$ is constant along the circles ${\mathcal C}_e(a$, $r)$,
for which $s_3$ is constant and $s_1^2 + s_2^2$ is constant.
Because $F_{X'}$ is a third--degree homogeneous polynomial
in the variables $s_1$, $s_2$, and $s_3$, we must have:
\begin{eqnarray*}
F_{X'}({\bf s}) & = & 3a\left(s_1^2 + s_2^2\right)s_3 + bs_3^3,
\end{eqnarray*}
for some real numbers $a$ and $b$. \\
We argue that we cannot have $a=0$. Indeed, if this was the case, then
	\[F_{X'}({\bf s}) = bs_3^3, \ \ \mbox{ and so } \ \ F_X({\bf t}) = b({\bf u}\cdot {\bf t})^3,  \]
	where ${\bf u} = (u_1,u_2,u_3)$ is the third row of $U$ (this also means that ${\bf u}=f_3$).
	Therefore,
	\[ \alpha_{i,j,k} = bu_iu_ju_k, \ \ \mbox{ for } \ \ (i,j,k) \in \lc{1,2,3}^3.\]
	However, this implies that either, one (or both) of the matrices $A_1$ and $A_2$ is zero, or else, they have proportional entries. In either case, the commutator $[A_1,A_2]=0$, which contradicts our hypothesis.\\
Since we have:
\begin{eqnarray*}
F_{X'}({\bf s}) & = & \sum_{i, j, k}\beta_{i, j, k}s_is_js_k,
\end{eqnarray*}
identifying the coefficients of $s_is_js_k$, for all possible values of $i$, $j$, and $k$ in $\{1$, $2$, $3\}$, we obtain:
\begin{eqnarray*}
\beta_{i, j, k} & := & \left\{
\begin{array}{cl}
a &  \mbox{if} \ (i,j,k) \ \mbox{ is a permutation of} \ (1,1,3)  \ {\rm or } \ (2,2,3) \\
b & \mbox{if} \  i = j = k = 3\\
0 &  {\rm otherwise}
\end{array}
\right..
\end{eqnarray*}
This implies that the matrices corresponding to the new three--dimensional $1$--Meixner random vector
$X' = (X_1'$, $X_2'$, $X_3')$ are:
 \begin{equation}
 \label{E:Aprime}
  A_1' = \begin{bmatrix}
0 & 0 & a \\
0& 0 & 0 \\
a & 0 & 0 \\
\end{bmatrix},
 \
A_2' = \begin{bmatrix}
0 & 0 & 0 \\
0& 0 & a \\
0 & a & 0 \\
\end{bmatrix},
 \
A_3' =
\begin{bmatrix}
a & 0 & 0 \\
0& a & 0 \\
0 & 0 & b \\
\end{bmatrix}.
\end{equation}
Therefore
\[C_{1,2}' = \ls{A_1',A_2'} =
\begin{bmatrix}
0 & a^2 & 0 \\
-a^2& 0 & 0 \\
0 & 0 & 0 \\
\end{bmatrix} \neq 0, \ \  \mbox{ as } \ \ a\neq 0.\]
On the other hand,  consider the commutator
\[C'_{2,3} = \ls{A_2',A_3'} =
\begin{bmatrix}
0 & 0 & 0 \\
0 & 0 & a(b-a) \\
0 & -a(b-a) & 0 \\
\end{bmatrix}.\]
Applying equation (\ref{n=1}) to $C'_{2,3}$ and ${\bf t} := e_2 = (0$, $1$, $0)$, we obtain:
\begin{eqnarray}
\left(C'_{2,3}e_2\right) \cdot \left(\left(0A_1' + 1A_2' + 0A_3'\right)e_2\right) & = & 0. \label{particular_n=1}
\end{eqnarray}
Since $C'_{2,3}e_2 = (0$, $0$, $-a(b - a)) = -a(b - a)e_3$ and $A_2'e_2 = (0$, $0$, $a) = ae_3$,
equation (\ref{particular_n=1}) becomes:
\begin{eqnarray}
-a^2(b - a) & = & 0. \label{a=b}
\end{eqnarray}
Because $a \neq 0$, we conclude from (\ref{a=b}),
that $b = a$, which also implies $C'_{2,3} = C'_{1,3} = 0$.
Therefore, the proof of this lemma is concluded.
\endproof

	In the next lemma we integrate the system \eqref{system} in the particular case of a canonical random vector (with associated cubic form as in Lemma~\ref{L:Fcanonical}).  Consequently,
	we  produce an explicit formula for the Laplace transform of its probability distribution.
	\begin{lemma}
		\label{L:phi_canonical}
		Let $X'$ be a random vector with associated  cubic form
		\[F_{X'} \lp{{\bf s}} =  3as_3(s_1^2+s_2^2)+as_3^3,\]
		for some $a\neq 0$. Then, the Laplace transform
		$\ds \vp(\vct{s})= E\ls{\exp \lp{{\bf s} \cdot X'}}$ is given
			by the formula
		\begin{equation}
		\label{Laplace_transform_formula}
		\vp(\vct{s})  =
		e^{-(1/a)s_3}  \lp{-a^2s_1^2 - a^2s_2^2 + (1 - as_3)^2}^{-1/(2a^2)},
		\end{equation}
		which is analytic in the interior of the cone
		\[ D:= \lc{ {\bf s} \in \reo^3 \ | \ |a|\sqrt{s_1^2 + s_2^2} < {1 - as_3}}.\]
	\end{lemma}
	\proof
Since $\varphi({\bf 0}) = 1 \neq 0$ and $\varphi$ is continuous at ${\bf 0}$, we can divide both sides of
each equation of the system (\ref{system}) by $\varphi({\bf s})$, for ${\bf s}$ in a neighborhood $V$ of ${\bf 0}$,
and conclude that the logarithm of the joint Laplace transform of $X'_1$, $X'_2$, $\dots$, $X'_d$:
\begin{eqnarray*}
\psi({\bf s}) & := & \ln\varphi({\bf s})
\end{eqnarray*}
satisfies the system of partial differential equations, written in matrix form as:
\begin{eqnarray*}
\nabla \psi({\bf s}) & = & M({\bf s}){\bf s},
\end{eqnarray*}
for all ${\bf s}$ in $V$, where:
\begin{eqnarray*}
M({\bf s}) & := & \left(I - B({\bf s})\right)^{-1},
\end{eqnarray*}
for:
\begin{eqnarray*}
B({\bf s}) & := & s_1A_1' + s_2A_2' + s_3A_3',
\end{eqnarray*}
with the coefficients $A_1', A_2', A_3'$ given by \eqref{E:Aprime}.
We have:
\[B({\bf s}) =
\begin{bmatrix}
as_3 & 0 & as_1 \\
0& as_3 & as_2 \\
as_1 & as_2 &  as_3\\
\end{bmatrix},\]
 and
\[
\begin{aligned}
& M({\bf s}) = \left(I - B({\bf s})\right)^{-1}  =\\
& \begin{bmatrix}
\frac{-a^2s_2^2+(1 - as_3)^2}{(1 - as_3)(-a^2s_1^2 - a^2s_2^2 + (1 - as_3)^2)} &
\frac{a^2s_1s_2}{(1 - as_3)(-a^2s_1^2 - a^2s_2^2+(1 - as_3)^2)} &
\frac{as_1}{-a^2s_1^2 - a^2s_2^2 + (1 - as_3)^2} \\
\frac{a^2s_1s_2}{(1 - as_3)(-a^2s_1^2 - a^2s_2^2 + (1 - as_3)^2)}   &
\frac{-a^2s_1^2 + (1 - as_3)^2}{(1 - as_3)(-a^2s_1^2 - a^2s_2^2 + (1 - as_3)^2)} &
\frac{as_2}{-a^2s_1^2 - a^2s_2^2 + (1 - as_3)^2}  \\
\frac{as_1}{-a^2s_1^2 - a^2s_2^2 + (1 - as_3)^2} &
\frac{as_2}{-a^2s_1^2 - a^2s_2^2 + (1 - as_3)^2} &
\frac{1 - as_3}{-a^2s_1^2 - a^2s_2^2 + (1 - as_3)^2}\\
\end{bmatrix}.
\end{aligned}
\]
We get
\[ M({\bf s})\vct{s} =
\begin{bmatrix}
\frac{s_1}{-a^2s_1^2 - a^2s_2^2 + (1 - as_3)^2}\\
\frac{s_2}{-a^2s_1^2 - a^2s_2^2 + (1 - as_3)^2}\\
\frac{as_1^2 + as_2^2 - as_3^2 + s_3}{-a^2s_1^2 - a^2s_2^2 + (1 - as_3)^2}\\
 \end{bmatrix},
\]
and from
$\nabla \psi(\vct{s}) = M({\bf s}) \vct{s}$ we obtain:
\[ \psi(\vct{s}) = -\frac{1}{2a^2} \ln\lp{-a^2s_1^2 - a^2s_2^2 + (1 - as_3)^2} - \frac{1}{a}s_3.\]
Since $\psi(\vct{s}) = \ln \vp(\vct{s})$, we conclude that the Laplace transform of $X'$ is:
\begin{eqnarray*}
\vp(\vct{s}) & = & e^{-(1/a)s_3} \cdot \left[-a^2s_1^2 - a^2s_2^2 + (1 - as_3)^2\right]^{-1/(2a^2)}, 
\end{eqnarray*}
and therefore the Lemma is proved.
\endproof

Our next step is to invert the Laplace transform \eqref{Laplace_transform_formula}, to obtain the joint probability distribution of $X_1'$, $X_2'$, and $X_3'$.\\
We introduce the following notations:
\begin{itemize}
	
\item We denote the Laplace transform by ${\mathcal L}$, and accordingly, its inverse by
${\mathcal L}^{-1}$.

\item For all $\vct{c} \in {\mathbb R}^3$, we denote by ${\mathcal E}_{\vct{c}}$, the exponential function:
\begin{eqnarray*}
{\mathcal E}_{\vct{c}}(\vct{x}) & := & e^{\vct{c} \cdot \vct{x}},
\end{eqnarray*}
for all $\vct{x} \in {\mathbb R}^3$.

\item For all $\vct{c} \in {\mathbb R}^3$, we denote by ${\mathcal T}_{\vct{c}}$, the translation operator that maps a function
$f$ into the function ${\mathcal T}_{\vct{c}}f$, defined by:
\begin{eqnarray*}
\left({\mathcal T}_{\vct{c}}f\right)(\vct{x}) & := & f(\vct{x} + \vct{c}),
\end{eqnarray*}
for all $\vct{x} \in {\mathbb R}^3$.

\item For all $c \in {\mathbb R} \setminus \{0\}$, we denote by ${\mathcal D}_c$, the dilation operator that maps a function
$f$ into the function ${\mathcal D}_cf$, defined by:
\begin{eqnarray*}
\left({\mathcal D}_cf\right)(\vct{x}) & := & f(c\vct{x}),
\end{eqnarray*}
for all $\vct{x} \in {\mathbb R}$.
\end{itemize}

We have the following properties:
\begin{itemize}

\item  For every function $f$ and $\vct{c} \in {\mathbb R}^3$,
\begin{eqnarray*}
{\mathcal L}\left({\mathcal T}_{\vct{c}}f\right) & = & {\mathcal E}_{-\vct{c}} \cdot \left({\mathcal L}f\right).
\end{eqnarray*}

\item  For every function $f$ and $\vct{c} \in {\mathbb R}^3$,
\begin{eqnarray*}
{\mathcal L}\left({\mathcal E}_{\vct{c}}f\right) & = & {\mathcal T}_{\vct{c}}\left({\mathcal L}f\right).
\end{eqnarray*}

\item For every function $f$ and $c \in {\mathbb R} \setminus \{0\}$,
\begin{eqnarray*}
{\mathcal L}\left({\mathcal D}_cf\right) & = & \frac{1}{|c|^3}{\mathcal D}_{1/c}\left({\mathcal L}f\right).
\end{eqnarray*}

\end{itemize}
Using these properties, it is not hard to see that:
\begin{eqnarray}
& \ & {\mathcal L}^{-1}\left[e^{-(1/a)s_3} \cdot \left[-a^2s_1^2 - a^2s_2^2 + (1 - as_3)^2\right]^{-1/(2a^2)}\right] \nonumber\\
& = & {\mathcal L}^{-1}\left\{{\mathcal D}_{a}
\left[e^{-(1/a^2)s_3} \cdot \left[-s_1^2 - s_2^2 + (1 - s_3)^2\right]^{-1/(2a^2)}\right]\right\} \nonumber\\
& = & \frac{1}{|a|^3}{\mathcal D}_{1/a}\left\{{\mathcal L}^{-1}\left[e^{-(1/a^2)s_3} \cdot \left[-s_1^2 - s_2^2 + (1 - s_3)^2\right]^{-1/(2a^2)}
\right]\right\} \nonumber\\
& = & \frac{1}{|a|^3}{\mathcal D}_{1/a}\left\{{\mathcal L}^{-1}\left[{\mathcal E}_{-(1/a^2)e_3} \cdot \left[-s_1^2 - s_2^2 + (1 - s_3)^2\right]^{-1/(2a^2)}
\right]\right\} \nonumber\\
& = & \frac{1}{|a|^3}{\mathcal D}_{1/a}\left\{{\mathcal T}_{(1/a^2)e_3}\left[{\mathcal L}^{-1}\left\{\left[-s_1^2 - s_2^2 + (s_3 - 1)^2\right]^{-1/(2a^2)}
\right\}\right]\right\} \nonumber\\
& = & \frac{1}{|a|^3}{\mathcal D}_{1/a}\left\{{\mathcal T}_{(1/a^2)e_3}\left[{\mathcal L}^{-1}
\left({\mathcal T}_{-e_3}\left\{\left(-s_1^2 - s_2^2 + s_3^2\right)^{-1/(2a^2)}\right\}\right)\right]\right\} \nonumber\\
& = & \frac{1}{|a|^3}{\mathcal D}_{1/a}\left\{{\mathcal T}_{(1/a^2)e_3}\left\{{\mathcal E}_{-e_3}
{\mathcal L}^{-1}\left[\left(-s_1^2 - s_2^2 + s_3^2\right)^{-1/(2a^2)}\right]\right\}\right\}. \label{inverse_Laplace_transform_interior}
\end{eqnarray}
It was shown in \cite{g75} (see also \cite{lw08} pages 6482--6483), that ${\mathcal L}^{-1}[(s_3^2 - s_1^2 - s_2^2)^{-1/(2a^2)}]$ is
a (positive) measure $\nu$ if and only if:
\begin{eqnarray*}
\frac{1}{2a^2} & \geq & \frac{3 - 2}{2},
\end{eqnarray*}
here the number $3$ from the right--hand side of the last inequality is the dimension $d$ of ${\mathbb R}^3$. The above inequality
is equivalent to:
\begin{eqnarray*}
|a| & \leq & 1.
\end{eqnarray*}
We will refer to a probability measure having the Laplace transform ${\mathcal L}({\bf s}) = (s_3^2 - s_1^2 - s_2^2)^{-1/(2a^2)}$,
as a {\em three-dimensional Gamma distribution}.\\
Moreover, for $|a| < 1$, the measure $\nu$ is absolutely continuous with respect to the Lebesgue measure $dx_1dx_2dx_3$ on ${\mathbb R}^3$.
Its Radon--Nikod\'ym derivative is:
\begin{eqnarray*}
g(x_1, x_2, x_3) & := & \frac{1}{2^{1/a^2 - 3/2}\Gamma_{\Omega}(1/(2a^2))}\left(x_3^2 - x_1^2 - x_2^2\right)^{1/(2a^2) - 3/2},
\end{eqnarray*}
where, for all $p > 1/2$
\begin{eqnarray*}
\Gamma_{\Omega}(p) & := & (2\pi)^{1/2}\Gamma(p)\Gamma\left(p - \frac{1}{2}\right),
\end{eqnarray*}
for all $\vct{x} = (x_1$, $x_2$, $x_3) \in \Omega$, where $\Omega$ is the open cone:
\begin{eqnarray*}
\Omega & := & \left\{\vct{x} \in {\mathbb R}^3 \mid x_3 > \sqrt{x_1^2 + x_2^2} \right\}.
\end{eqnarray*}
We have also included the computation of the Laplace transform of these functions in the Appendix.\\
Let us define $p := 1/(2a^2) > 1/2$. Then the joint probability distribution $\mu$ of $X_1'$, $X_2'$, $X_3'$ is absolutely
continuous with respect to the Lebesgue measure on ${\mathbb R}^3$, and its density function is
according to formula (\ref{inverse_Laplace_transform_interior}):
\begin{eqnarray}
f(x_1, x_2, x_3) & := & \frac{1}{|a|^3}{\mathcal D}_{1/a}\left\{{\mathcal T}_{(1/a^2)e_3}\left\{{\mathcal E}_{-e_3}
g(x_1, x_2, x_3)\right\}\right\} \nonumber\\
& = & \frac{1}{|a|^3}{\mathcal D}_{1/a}\left\{{\mathcal T}_{2pe_3}\left\{e^{-x_3}g(x_1, x_2, x_3)
\right\}\right\} \nonumber\\
& = & \frac{1}{|a|^3}{\mathcal D}_{1/a}\left\{e^{-(x_3 + 2p)}g(x_1, x_2, x_3 + 2p)\right\} \nonumber\\
& = & \frac{e^{-2p}}{|a|^3}e^{-x_3/a}g\left(\frac{x_1}{a}, \frac{x_2}{a}, \frac{x_3}{a} + 2p\right) \nonumber\\
& = & \frac{e^{-1/a^2}}{|a|^3}e^{-x_3/a}\frac{1}{2^{1/a^2 - 3/2}\Gamma_{\Omega}(1/(2a^2))} \nonumber\\
&     & \left[\left(\frac{x_3}{a} + \frac{1}{a^2}\right)^2 - \frac{x_1^2}{a^2} - \frac{x_2^2}{a^2}\right]^{1/(2a^2) - 3/2} \nonumber\\
& = & C_ae^{-(x_3/a + 1/a^2)}\left[\left(\frac{x_3}{a} + \frac{1}{a^2}\right)^2 - \frac{x_1^2}{a^2} - \frac{x_2^2}{a^2}\right]^{1/(2a^2) - 3/2},
\label{density_formula}
\end{eqnarray}
for all $\vct{x} = (x_1$, $x_2$, $x_3)$ inside a cone $\Omega_a$,
where $C_a$ is the positive constant:
\begin{eqnarray*}
C_a & := & \frac{1}{2^{1/a^2 - 3/2}|a|^3\Gamma_{\Omega}(1/(2a^2))}.
\end{eqnarray*}
The presence of the factor $e^{-x_3/a}$ in formula (\ref{density_formula}),
ensures the fact that $\mu$ has finite moments of all orders inside the (shifted) cone $\Omega_a$, where:
\begin{eqnarray*}
\Omega_a & := & \left\{(x_1, x_2, x_3) \in {\mathbb R}^3 ~\left|~
\left(\frac{x_3}{a} + \frac{1}{a^2}\right)^2 > \frac{x_1^2}{a^2} + \frac{x_2^2}{a^2}, \quad \frac{x_3}{a} + \frac{1}{a^2} > 0\right.\right\}.
\end{eqnarray*}
For $|a| = 1$, formula (\ref{Laplace_transform_formula}) becomes:
\begin{eqnarray*}
\varphi\left(\vct{s}\right) & = & \frac{e^{\pm s_3}}{\sqrt{(1 \pm s_3)^2 - s_1^2 - s_2^2}}.
\end{eqnarray*}
The case when $a = -1$ can be obtained from the case $a = 1$, by applying a dilation of factor $c = -1$, since:
\begin{eqnarray*}
\varphi_{-1}\left(\vct{s}\right) & = & \frac{e^{s_3}}{\sqrt{(1 + s_3)^2 - s_1^2 - s_2^2}} \nonumber\\
& = & \frac{e^{-(-s_3)}}{\sqrt{[1 - (-s_3)]^2 - (-s_1)^2 - (-s_2)^2}} \nonumber\\
& = & \varphi_{+1}\left(-\vct{s}\right).
\end{eqnarray*}
Thus, we have:
\begin{eqnarray*}
{\mathcal L}^{-1}\left(\varphi_{-1}\right) & = & {\mathcal L}^{-1}\left({\mathcal D}_{-1}\varphi_{+1}\right) \nonumber\\
& = & {\mathcal D}_{-1}{\mathcal L}^{-1}\left(\varphi_{+1}\right).
\end{eqnarray*}
Therefore, if we find the measure of $\mu_{+1}$, corresponding to $a = +1$, and its support is $D_{+1}$,
then the measure $\mu_{-1}$, corresponding to $a = -1$, will be supported by $D_{-1} := -D_{+1}$, and for all Borel subsets $B$
of $D_{-1}$, we have $\mu_{-1}(B) = \mu_{+1}(-B)$.\\
The probability measure $\mu_{+1}$, corresponding to $a = +1$, is supported on a cone (a two--dimensional manifold), and its construction
is described in the Appendix.\\
{\bf Case 2.} \ If $[A_1$, $A_2] = [A_2$, $A_3] = [A_3$, $A_1] = 0$. Since all the matrices $A_1$, $A_2$, and $A_3$
are symmetric and commute, they can be diagonalized in the same basis $\{f_1$, $f_2$, $f_3\}$ of ${\mathbb R^3}$.
Let $U : {\mathbb R}^3 \to {\mathbb R}^3$ be the orthogonal
 linear transformation that maps the standard basis
$\{e_1$, $e_2$, $e_3\}$ of ${\mathbb R}^3$ into $\{f_1$, $f_2$, $f_3\}$. We identify $U$ with its matrix
$\{u_{i, j}\}_{1 \leq i, j \leq d}$. We have $U^{-1} = U^T$ (the transpose of $U$), and
for all $r \in \{1$, $2$, $3\}$:
\begin{eqnarray*}
UA_rU^T & = & D_r,
\end{eqnarray*}
where $D_r$ is a $3 \times 3$ diagonal matrix.\\
Let us define the following random variables:
\begin{eqnarray*}
X_1' & := & \sum_{j = 1}^3u_{1, j}X_j,\\
X_2' & := & \sum_{j = 1}^3u_{2, j}X_j,\\
X_3' & := & \sum_{j = 1}^3u_{3, j}X_j.
\end{eqnarray*}
Then, for all $i \in \{1$, $2$, $3\}$, we have:
\begin{eqnarray*}
X_i & = & \sum_{j = 1}^3u_{i, j}^TX_j' \nonumber\\
& = & \sum_{j = 1}^3u_{j, i}X_j'.
\end{eqnarray*}
Since $(X_1$, $X_2$, $X_3)$ is a non--degenerate $1$-Meixner random vector, and $U$ is an invertible linear transformation,
$(X_1'$, $X_2'$, $X_3')$ is also a non-degenerate $1$-Meixner random vector. Moreover, the joint semi--annihilation
operators of $X_1'$, $X_2'$, and $X_3'$ are, for all $i \in \{1$, $2$, $3\}$:
\begin{eqnarray*}
U_i' & := & \sum_{j = 1}^3u_{i, j}U_j.
\end{eqnarray*}
For all $i$ and $j $ in $\{1$, $2$, $3\}$, such that $i \neq j$, we have:
\begin{eqnarray*}
\left[U_i', X_j'\right] & = & \left[\sum_{k = 1}^3u_{i, k}U_k, \sum_{l = 1}^3u_{j, l}X_l\right] \nonumber\\
& = & \sum_{k, l}u_{i, k}u_{j, l}\left[U_k, X_l\right] \nonumber\\
& = & \sum_{k, l}u_{i, k}u_{j, l}\left(\sum_r\alpha_{k,l,r}X_r + \delta_{i, j}I\right) \nonumber\\
& = & \sum_{k, l, r}u_{i, k}\alpha_{r, k, l}u_{l, j}^{T}\sum_su_{s, r}X_s' \nonumber\\
& = & \sum_{r, s}u_{s, r}\left(\sum_{k, l}u_{i, k}\alpha_{r, k, l}u_{l, j}^T\right)X_s' \nonumber\\
& = & \sum_{r, s}u_{s, r}\left(UA_rU^T\right)_{i, j}X_s' \nonumber\\
& = & \sum_{r, s}u_{s, r}\left(D_r\right)_{i, j}X_s' \nonumber\\
& = & 0,
\end{eqnarray*}
since $D_r$ is a diagonal matrix and $i \neq j$.\\
Since for all $i \neq j$, we have $[U_i'$, $X'_j] = 0$, it follows from Theorem 4.6 from \cite{ps19}, that
the joint probability distribution of $(X_1'$, $X_2'$, $X_3')$ is polynomially factorisable. That means,
for all $i$, $j$, and $k \in {\mathbb N} \cup \{0\}$, we have:
\begin{eqnarray*}
E\left[X_1'^iX_2'^jX_3'^k\right] & = & E\left[X_1'^i\right]E\left[X_2'^j\right]E\left[X_3'^k\right].
\end{eqnarray*}
That means, from the point of view of moments $X_1'$, $X_2'$, and $X_3'$ behave like three independent random variables.
Since the new coefficients $\alpha_{i,j,k}' = 0$, for all $i \neq j$, permuting the indexes, we obtain, that,
for all $k \in \{1$, $2$, $3\}$, we have:
\begin{eqnarray*}
\alpha_{k,i,j}' & = & 0,
\end{eqnarray*}
for all $i \neq j$. Thus, if we choose $i := k$ and $j \neq k$, we have:
\begin{eqnarray*}
\alpha_{k, k, j}' & = & 0.
\end{eqnarray*}
That means, for all $k \in \{1$, $2$, $3\}$, we have:
\begin{eqnarray*}
\left[U_k', X_k'\right] & = & \sum_{j = 1}^3\alpha_{k, k, j}'X_j' + c_k'I \nonumber\\
& = & \alpha_{k, k, k}'X_k' + c_k'I.
\end{eqnarray*}
The last equation shows that individually each random variable $X_1'$, $X_2'$, and $X_3'$ is a $1$-Meixner random
variable. It was shown in \cite{ps14}, that the $1$-Meixner random variables are up to a re-scaling and translation
Gamma or Gaussian random variables. Since the joint moments of $X_1$, $X_2$, and $X_3$ can be written as products
of the corresponding individual moments of $X_1$, $X_2$, and $X_3$, and because the moment problem for $1$-Meixner
random vectors is uniquely solvable (due to the estimates that we obtained in
Lemma \ref{moment_estimate_lemma}), we conclude that the joint probability distribution $\mu$ of $X_1'$, $X_2'$, and $X_3'$
is the product of the individual probability distributions $\mu_1$, $\mu_2$, and $\mu_3$ of $X_1'$, $X_2'$, and
$X_3'$, respectively. Thus $X_1'$, $X_2'$, and $X_3'$ are three independent re-scaled and shifted Gamma or Gaussian
random variables.\\
We conclude our discussion with the following theorem:
\begin{theorem}
	\label{T:main}
A three--dimensional random vector $(X_1$, $X_2$, $X_3)$, having finite joint moments of all orders, is a non--degenerate $1$--Meixner
random vector if and only if there exists an invertible affine transformation from
${\mathbb R}^3$ to ${\mathbb R}^3$, denoted
$$ (X_1, X_2, X_3) \to (X_1', X_2', X_3'),$$
such that either:
\begin{itemize}

\item $X_1'$, $X_2'$, and $X_3'$ are independent Gamma or Gaussian random variables

\end{itemize}

or

\begin{itemize}

\item the joint probability distribution of $(X_1'$, $X_2'$, $X_3')$ is a three--dimensional Gamma distribution.

\end{itemize}

\end{theorem}

\section{Appendix}
\label{S:app}

\begin{proposition}
\label{P:Ltinv}
Let $\mu$ be the measure on ${\mathbb R}^3$, that is absolutely continuous with respect
to the Lebesgue measure $dx_1dx_2dx_3$ on ${\mathbb R}^3$, and whose Radon-Nikod\'ym derivative is:
\begin{eqnarray*}
\frac{d\mu(x)}{dxdydz} & = & \left(x_3^2 - x_1^2 - x_2^2\right)^p1_\Omega(x_1, x_2, x_3),
\end{eqnarray*}
where:
\begin{eqnarray*}
\Omega & := & \left\{(x_1, x_2, x_3) \in {\mathbb R}^3 \mid x_1^2 + x_2^2 < x_3^2,
\ x_3 > 0 \right\},
\end{eqnarray*}
is the open upper cone with vertex at the origin, axis of symmetry the positive $x_3$-axis, whose generator
and axis of symmetry make an angle $\varphi$, of measure $m(\varphi) = 45^{\circ}$,
and $1_\Omega$ denotes the
characteristic function of $\Omega$.
Then, for $p > -1$,  the Laplace transform of $\mu$ (or equivalently, of its Radon-Nikod\'ym derivative
with respect to the Lebesgue measure) is:
\begin{eqnarray*}
E_{\mu}\left[\exp\left(\vct{t} \cdot X\right)\right] & = & 2\pi\Gamma(2p + 2)\left(t_3^2 - t_1^2 - t_2^2\right)^{-p - (3/2)}.
\end{eqnarray*}
\end{proposition}
\proof
Let $\vct{t} = (t_1$, $t_2$, $t_3) \in {\mathbb R}^3$. We have:
\begin{eqnarray*}
\varphi\left(\vct{t}\right) & = & \int_De^{t_1x_1}e^{t_2x_2}e^{t_3x_3}\left(x_3^2 - x_1^2 - x_2^2\right)^p
dx_1dx_2dx_3.
\end{eqnarray*}
Let us move to spherical coordinates:
\begin{eqnarray*}
\left\{
\begin{array}{rrr}
x_1 & = & r\sin(\varphi)\cos(\theta)\\
x_2 & = & r\sin(\varphi)\sin(\theta)\\
x_3 & = & r\cos(\varphi),
\end{array}
\right.
\end{eqnarray*}
where $r \in (0$, $\infty)$, $\theta \in [0$, $2\pi)$, $\varphi \in [0$, $\pi/4)$.
The absolute value of the Jacobian is:
\begin{eqnarray*}
|J(r, \theta, \varphi)| & = & r^2\sin(\varphi).
\end{eqnarray*}
Thus, we have:
\begin{eqnarray*}
\varphi\left(\vct{t}\right) & = & \int_0^{\pi/4}   \sin(\varphi)\left[\int_0^{\infty}e^{t_3r\cos(\varphi)}\left(r^2\cos(2\varphi)\right)^pr^2 \right. \\
 &  &
\left. \left[\int_0^{2\pi}e^{r\sin(\varphi)(t_1\cos(\theta) + t_2\sin(\theta))}d\theta\right]dr\right]d\varphi.\\
\end{eqnarray*}
The order of integration that we choose is: first with respect to $\theta$, second with respect to $r$, and third
we respect to $\varphi$.\\
Let us compute first the innermost integral, with respect to $\theta$, for $r$ and $\varphi$ fixed. In the exponent
of that integrand we write the superposition of waves as only one wave, namely:
\begin{eqnarray*}
t_1\cos(\theta) + t_2\sin(\theta) & = & \sqrt{t_1^2 + t_2^2}\left[\cos(\theta)\frac{t_1}{\sqrt{t_1^2 + t_2^2}}
+ \sin(\theta)\frac{t_2}{\sqrt{t_1^2 + t_2^2}} \right] \nonumber\\
& = & \sqrt{t_1^2 + t_2^2}\left[\cos(\theta)\cos(\tau) + \sin(\theta)\sin(\tau)\right] \nonumber\\
& = & \sqrt{t_1^2 + t_2^2}\cos(\theta - \tau),
\end{eqnarray*}
where $\tau \in [0$, $2\pi)$ is the only angle such that:
\begin{eqnarray*}
\left\{
\begin{array}{ccc}
\cos(\tau) & = & \frac{t_1}{\sqrt{t_1^2 + t_2^2}}\\
\sin(\tau) & = & \frac{t_2}{\sqrt{t_1^2 + t_2^2}}
\end{array}
\right..
\end{eqnarray*}
Therefore, the most inner integral becomes:
\begin{eqnarray*}
I_1 & := & \int_0^{2\pi}e^{r\sin(\varphi)\sqrt{t_1^2 + t_2^2}\cos(\theta - \tau)}d\theta.
\end{eqnarray*}
Since the integrand is a periodic function of period $2\pi$, we can integrate on any interval of
length $2\pi$. So, we can integrate on the interval $[\tau$, $\tau + 2\pi)$, obtaining after
the change of variable $\theta \mapsto \theta + \tau$:
\begin{eqnarray*}
I_1 & = & \int_0^{2\pi}e^{r\sin(\varphi)\sqrt{t_1^2 + t_2^2}\cos(\theta)}d\theta.
\end{eqnarray*}
We use now the Taylor series expansion of the exponential function:
\begin{eqnarray*}
e^z & = & \sum_{n = 0}^{\infty}\frac{z^n}{n!}.
\end{eqnarray*}
Since this is an entire series (radius of convergence $R = \infty$), according to Weierstrass Theorem,
it converges uniformly on any compact $K$. Thus, we can interchange the series and integral, obtaining:
\begin{eqnarray*}
I_1 & = & \int_0^{2\pi}e^{r\sin(\varphi)\sqrt{t_1^2 + t_2^2}\cos(\theta)}d\theta \nonumber\\
& = & \int_0^{2\pi}\sum_{n = 0}^{\infty}\frac{\left(r\sin(\varphi)\sqrt{t_1^2 + t_2^2}\cos(\theta)\right)^n}{n!}d\theta \nonumber\\
& = & \sum_{n = 0}^{\infty}\frac{\left(r\sin(\varphi)\sqrt{t_1^2 + t_2^2}\right)^n}{n!}\int_0^{2\pi}\cos^n(\theta)d\theta.
\end{eqnarray*}
For all $n \geq 0$, let us define:
\begin{eqnarray*}
J_n & := & \int_0^{2\pi}\cos^n(\theta)d\theta.
\end{eqnarray*}
Then $J_0 = 2\pi$, $J_1 = 0$, and for all $n \geq 2$, integrating by parts, we obtain:
\begin{eqnarray*}
J_n & = & \int_0^{2\pi} \cos^{n - 1}(\theta)\cos(\theta)d\theta \nonumber\\
& = & \int_0^{2\pi} \cos^{n - 1}(\theta)d\left(\sin(\theta)\right) \nonumber\\
& = & \cos^{n - 1}(\theta)\sin(\theta)|_0^{2\pi} + (n - 1)\int_0^{2\pi}\cos^{n - 2}(\theta)\sin^2(\theta)d\theta \nonumber\\
& = & 0 + (n - 1)\int_0^{2\pi}\cos^{n - 2}(\theta)\left[1 - \cos^2(\theta)\right]d\theta \nonumber\\
& = & (n - 1)J_{n - 2} - (n - 1)J_n.
\end{eqnarray*}
We obtain from here, that:
\begin{eqnarray*}
J_n & = & \frac{n - 1}{n}J_{n - 2},
\end{eqnarray*}
for all $n \geq 1$. Iterating this recursive relation, we get:
\begin{eqnarray*}
J_{2n + 1} & = & 0
\end{eqnarray*}
and
\begin{eqnarray*}
J_{2n} & = & \frac{(2n - 1)!!}{(2n)!!} 2\pi\nonumber\\
& = & \frac{(2n - 1)!!}{2^nn!}2\pi.
\end{eqnarray*}
Thus, we obtain:
\begin{eqnarray*}
\int_0^{2\pi}e^{r\sqrt{t_1^2 + t_2^2}\cos(\theta)}d\theta & = & \sum_{n = 0}^{\infty}
\frac{r^{2n}(t_1^2 + t_2^2)^n\sin^{2n}(\varphi)}{(2n)!} \cdot \frac{(2n - 1)!!}{2^nn!} 2\pi
\nonumber\\
& = & \sum_{n = 0}^{\infty}
\frac{r^{2n}(t_1^2 + t_2^2)^n\sin^{2n}(\varphi)}{2^{2n}(n!)^2}2\pi.
\end{eqnarray*}
We compute now the second integral, with respect to $r$, for a fixed $\varphi$.
\begin{eqnarray*}
I_2 & := & \sum_{n = 0}^{\infty}\frac{(t_1^2 + t_2^2)^n\sin^{2n}(\varphi)}{2^{2n}(n!)^2}2\pi
\int_0^{\infty}r^{2n + 2p + 2}e^{t_3r\cos(\varphi)}dr.
\end{eqnarray*}
Here, for convergence, we must assume that $t_3 < 0$,
and make the change of variable $s = r|t_3|\cos(\varphi)$,
$dr = \frac{1}{|t_3|\cos(\varphi)}ds$. We obtain:
\begin{eqnarray*}
I_2 & = & \sum_{n = 0}^{\infty}\frac{(t_1^2 + t_2^2)^n\sin^{2n}(\varphi)}{2^{2n}(n!)^2}2\pi
\int_0^{\infty}s^{2n + 2p + 2}e^{-s}\frac{1}{|t_3|^{2n + 2p + 3}\cos^{2n + 2p + 3}(\varphi)}ds \nonumber\\
& = & \sum_{n = 0}^{\infty}\frac{(t_1^2 + t_2^2)^n\sin^{2n}(\varphi)}{2^{2n}(n!)^2|t_3|^{2n + 2p + 3}
\cos^{2n + 2p + 3}(\varphi)}2\pi\Gamma(2n + 2p + 3).
\end{eqnarray*}
Finally, we compute the last integral, with respect to $\varphi$, which is the Laplace transform of $\mu$.
We have:
\begin{eqnarray*}
\varphi\left(\vct{t}\right) & = & \sum_{n = 0}^{\infty}\frac{(t_1^2 + t_2^2)^n}{2^{2n}(n!)^2|t_3|^{2n + 2p + 3}}
2\pi\Gamma(2n + 2p + 3)\\
 & & \int_0^{\pi/4}\frac{\sin^{2n + 1}(\varphi)}{\cos^{2n + 2p + 3}(\varphi)}\cos^p(2\varphi)
d\varphi \nonumber\\
& = & \sum_{n = 0}^{\infty}\frac{(t_1^2 + t_2^2)^n}{2^{2n}(n!)^2|t_3|^{2n + 2p + 3}}
2\pi\Gamma(2n + 2p + 3)\\
& & \int_0^{\pi/4}\frac{\sin^{2n}(\varphi)}{\cos^{2n + 2p + 3}(\varphi)}\cos^p(2\varphi)
\sin(\varphi)d\varphi.
\end{eqnarray*}
We make now the change of variable $u = \cos(\varphi)$, $du = -\sin(\varphi)d\varphi$. We obtain:
\begin{eqnarray*}
\varphi\left(\vct{t}\right) & = & \sum_{n = 0}^{\infty}\frac{(t_1^2 + t_2^2)^n}{2^{2n}(n!)^2|t_3|^{2n + 2p + 3}}
2\pi\Gamma(2n + 2p + 3)\int_{\sqrt{2}/2}^{1}\frac{(1 - u^2)^n}{u^{2n + 2p + 3}}\left(2u^2 - 1\right)^pdu.
\end{eqnarray*}
We make now the change of variable $v = u^2$, $dv = 2udu$. We obtain:
\begin{eqnarray*}
\varphi\left(\vct{t}\right) & = & \frac{1}{2}\sum_{n = 0}^{\infty}\frac{(t_1^2 + t_2^2)^n}{2^{2n}(n!)^2|t_3|^{2n + 2p + 3}}
2\pi\Gamma(2n + 2p + 3)\int_{1/2}^{1}\frac{(1 - v)^n(2v - 1)^p}{v^{n + p + 2}}dv \nonumber\\
& = & \frac{1}{2}\sum_{n = 0}^{\infty}\frac{(t_1^2 + t_2^2)^n}{2^{2n}(n!)^2|t_3|^{2n + 2p + 3}}
2\pi\Gamma(2n + 2p + 3)\int_{1/2}^{1}\left(\frac{1 - v}{v}\right)^n\left(\frac{2v - 1}{v}\right)^p\frac{1}{v^2}dv.
\end{eqnarray*}
Finally, we make the change of variable $w = (1 - v)/v$, $dw = -(1/v^2)dv$. We obtain:
\begin{eqnarray*}
\varphi\left(\vct{t}\right)
& = & \frac{1}{2}\sum_{n = 0}^{\infty}\frac{(t_1^2 + t_2^2)^n}{2^{2n}(n!)^2|t_3|^{2n + 2p + 3}}
2\pi\Gamma(2n + 2p + 3)\int_0^1w^n(1 - w)^pdw \nonumber\\
& = & \frac{1}{2}\sum_{n = 0}^{\infty}\frac{(t_1^2 + t_2^2)^n}{2^{2n}(n!)^2|t_3|^{2n + 2p + 3}}
2\pi\Gamma(2n + 2p + 3)B(n + 1, p + 1),
\end{eqnarray*}
where $B$ is the Euler beta function. Applying the formula:
\begin{eqnarray*}
B(n + 1, p + 1) & = & \frac{\Gamma(n + 1)\Gamma(p + 1)}{\Gamma(n + p + 2)},
\end{eqnarray*}
we have:
\begin{eqnarray*}
\varphi\left(\vct{t}\right) & = & \frac{1}{2}\sum_{n = 0}^{\infty}\frac{(t_1^2 + t_2^2)^n}{2^{2n}(n!)^2|t_3|^{2n + 2p + 3}}
2\pi\Gamma(2n + 2p + 3)\frac{\Gamma(n + 1)\Gamma(p + 1)}{\Gamma(n + p + 2)} \nonumber\\
& = & \frac{1}{2}\sum_{n = 0}^{\infty}\frac{(t_1^2 + t_2^2)^n}{2^{2n}(n!)^2|t_3|^{2n + 2p + 3}}
2\pi\Gamma\left(2\left(n + p + \frac{3}{2}\right)\right)
\frac{n!\Gamma(p + 1)}{\Gamma(n + p + 2)}.
\end{eqnarray*}
Using now Legendre duplication formula for the Gamma function:
\begin{eqnarray*}
\Gamma(2z) & = & \frac{2^{2z - 1}\Gamma(z)\Gamma(z + (1/2))}{\sqrt{\pi}},
\end{eqnarray*}
we obtain:
\begin{eqnarray*}
\varphi\left(\vct{t}\right) & & \\
& = &  \pi\sum_{n = 0}^{\infty}\frac{(t_1^2 + t_2^2)^n}{2^{2n}n!|t_3|^{2n + 2p + 3}}\frac{2^{2n + 2p + 2}\Gamma(n + p + (3/2))
\Gamma(n + p + 2)}{\sqrt{\pi}}  \frac{\Gamma(p + 1)}{\Gamma(n + p + 2)} \nonumber\\
\end{eqnarray*}
\begin{eqnarray*}
& = &\pi \frac{2^{2p + 2}}{|t_3|^{2p + 3}}\sum_{n = 0}^{\infty}\frac{(t_1^2 + t_2^2)^n}{n!|t_3|^{2n}}\frac{\Gamma(n + p + (3/2))}{\sqrt{\pi}}
\Gamma(p + 1) \nonumber\\
\end{eqnarray*}
\begin{eqnarray*}
& = & \pi\frac{2^{2p + 2}}{|t_3|^{2p + 3}}\sum_{n = 0}^{\infty}\frac{(t_1^2 + t_2^2)^n}{n!|t_3|^{2n}}\\
 & & \frac{(n + p + (1/2))(n - 1 + p + (1/2)) \cdots
(1 + p + (1/2))\Gamma(1 + p + (1/2))}{\sqrt{\pi}} \nonumber\\
& \ & \Gamma(p + 1) \nonumber\\
\end{eqnarray*}
\begin{eqnarray*}
& = & \pi\frac{2^{2p + 2}}{|t_3|^{2p + 3}}\\
& & \sum_{n = 0}^{\infty}(-1)^n\frac{(-p - (3/2))(-p - (3/2) - 1) \cdots (-p - (3/2) - (n - 1))}{n!} \\
& & \frac{(t_1^2 + t_2^2)^n}{|t_3|^{2n}}\nonumber
\frac{\Gamma(p + 1)\Gamma(p + 1 + (1/2))}{\sqrt{\pi}}.
\end{eqnarray*}
Using now the binomial formula:
\begin{eqnarray*}
(1 + x)^r & = & \sum_{n = 0}^{\infty}\frac{r(r - 1) \cdots (r - n + 1)}{n!}x^n,
\end{eqnarray*}
for $|x| < 1$, and the duplication formula again, we obtain:
\begin{eqnarray*}
\varphi\left(\vct{t}\right)
& = &\pi \frac{2^{2p + 2}}{|t_3|^{2p + 3}}\left[1 - \frac{t_1^2 + t_2^2}{t_3^2}\right]^{-p - (3/2)}\frac{\Gamma(2p + 2)}{2^{2p + 1}} \nonumber\\
& = & 2\pi\Gamma(2p + 2)\left(t_3^2 - t_1^2 - t_2^2\right)^{-p - (3/2)}.
\end{eqnarray*}
\endproof
The case $a = 1$ corresponds to $p=-1$ which is not covered by Proposition~\ref{P:Ltinv} above.
However, in this situation we have the simpler
 \begin{proposition}
 \label{P:Ltcyl}
 Let $\nu\lp{\vct{x}}$ be the two dimensional measure on $\partial \Omega$
  (the surface of the cone with equation
 $\sqrt{x_1^2+x_2^2} = x_3$) obtained as the push--forward of the surface measure on the
 cylinder given by parametrization
 \[ (r, \theta) \to (\cos(\theta), \sin(\theta), r),  \ r> 0, \ \theta \in [0, 2\pi), \]
 via the map
   \[ (\cos(\theta), \sin(\theta), r) \to ( r\cos(\theta), r\sin(\theta), r). \]
   Then, for  $\vct{t}\in -\Omega$ the Laplace transform
   \[ \cL \nu \lp{\vct{t}} = \int_{\partial \Omega} e^{\vct{t}\cdot \vct{x}} \,d\nu\!\lp{\vct{x}}, \]
   is given by
   \[ \cL \nu \lp{\vct{t}} =    \frac{2\pi}{\sqrt{t_3^2-\lp{t_1^2 + t_2^2}}}. \]
  \end{proposition}
 \proof
Using $\vct{x}= ( r\cos(\theta), r\sin(\theta), r)$, we have
 \[ \cL \nu \lp{\vct{t}} = \int_0^\infty \int_0^{2\pi} e^{t_3r}
 e^{t_1r\cos(\theta) + t_2r\sin(\theta)} \, d\theta\,dr. \]
 Therefore
  \[ \cL \nu \lp{\vct{t}} = \int_0^\infty e^{t_3r}  \int_0^{2\pi}
 e^{r\sqrt{t_1^2 +t_2^2} \cos\lp{\theta -\tau}} \, d\theta\,dr, \]
 where $\tau \in [0, 2\pi)$ is the only angle such that
\[
\left\{
\begin{array}{ccc}
\cos (\tau)  & = & \frac{t_1}{\sqrt{t_1^2 + t_2^2}}\\
\sin (\tau) & = & \frac{t_2}{\sqrt{t_1^2 + t_2^2}}
\end{array}
\right..
\]
As in Proposition~\ref{P:Ltinv}, we have
\begin{eqnarray*}
I_1:=\int_0^{2\pi}e^{r\sqrt{t_1^2 + t_2^2}\cos(\theta)}d\theta & = & \sum_{n = 0}^{\infty}
\frac{r^{2n}(t_1^2 + t_2^2)^n}{(2n)!}\frac{(2n - 1)!!}{2^{n}n!}2\pi \nonumber\\
& = & \sum_{n = 0}^{\infty}
\frac{r^{2n}(t_1^2 + t_2^2)^n}{2^{2n}(n!)^2}2\pi.
\end{eqnarray*}
We compute now the second integral, with respect to $r$.
\begin{eqnarray*}
I_2 & := & \sum_{n = 0}^{\infty}\frac{(t_1^2 + t_2^2)^n}{2^{2n}(n!)^2}2\pi
\int_0^{\infty}r^{2n}e^{t_3r}dr.
\end{eqnarray*}
Here, for convergence at infinity, we must assume that $t_3 < 0$.
Making the change of variable $s = r|t_3|$,
$dr = \frac{1}{|t_3|}ds$. We obtain:
\begin{eqnarray*}
I_2 & = & \sum_{n = 0}^{\infty}\frac{(t_1^2 + t_2^2)^n}{2^{2n}(n!)^2}2\pi
\int_0^{\infty}s^{2n}e^{-s}\frac{1}{|t_3|^{2n + 1}}ds \nonumber\\
& = & \sum_{n = 0}^{\infty}
\frac{(t_1^2 + t_2^2)^n}{2^{2n}(n!)^2|t_3|^{2n + 1}}2\pi(2n)!.
\end{eqnarray*}
Therefore, for  $\frac{t_1^2 + t_2^2}{t_3^2} <1$, we have
\[
I_2  =  \frac{2\pi}{|t_3|} \frac{1}{\sqrt{1-4\frac{t_1^2 + t_2^2}{4t_3^2}}}
 =  \frac{2\pi}{\sqrt{t_3^2-\lp{t_1^2 + t_2^2}}}.\]
\endproof


\begin{thebibliography}{2}
\bibitem{aks03} Accardi, L., Kuo, H.--H. and Stan, A. I.:
Characterization of probability measures through the canonically
associated interacting Fock spaces, {\it Infin. Dimens. Anal.
Quantum Probab. Relat. Top.} {\bf 7} No 4 (2004) 485--505.

\bibitem{aks04} Accardi, L., Kuo, H.--H. and Stan, A. I.:
Moments and commutators of probability measures, {\it Infin. Dimens.
Anal. Quantum Probab. Relat. Top.} {\bf 10} No 4 (2007) 591--612.

\bibitem{an02} Accardi, L., Nahni, M.:  Interacting Fock Spaces
and Orthogonal Polynomials in several variables {\it Non-Commutativity,
Infinite-Dimensionality and Probability at the Crossroads} {\bf 16} (2002) 192--205

\bibitem{g75} Gindikin, S.: Invariant generalized functions in homogenous domains,
{\it J. Functional Anal. Appl.}
{\bf 9} (1975) 50--52.

\bibitem{lw08} Letac, G., Wesolowski, J.: Laplace transforms which are negative powers of quadratic polynomials,
{\it Trans. Amer. Math. Soc.} {\bf 360} No 12 (2008) 6475--6496.

\bibitem{m34} Meixner, J.: Orthogonale Polynomsysteme mit einer
besonderen Gestalt der erzeugenden Funktion, {\it J. London Math.
Soc.} {\bf 9} (1934) 6--13.

\bibitem{ps14} Popa, G., Stan, A. I.: Gamma distributed random variables
and their semi--quantum operators, in: {\it J. Phys.: Conf. Ser. 563 012029}
doi:10.1088/1742-6596/563/1/012029 (2014).

\bibitem{ps15} Popa, G.,  Stan, A. I.: Two-dimensional 1-Meixner random vectors
and their semi-quantum operators {\it Communications on Stochastic Analysis (COSA)}
{\bf 9} No 4 (2015) 425--455.

\bibitem{ps16} Popa, G., Stan, A. I.: 2-Meixner random variables and semi-quantum operators
{\it J. Phys.: Conf. Ser.} {\bf 819} (2017) 012002.

\bibitem{ps19} Popa, G., Stan, A. I.: A characterization of probability measures in
terms of semi-quantum operators,
{\it Infin. Dimens. Anal. Quantum Probab. Relat. Top.} {\bf 22} No 02 (2019)
https://doi.org/10.1142/S0219025719500097.

\end{thebibliography}
\end{document}